\documentclass[12pt,oneside,english]{amsart}

\usepackage{cpold}

\title{Polynomiality of the Stirling coefficients of $c\big(\!\Pol^d(\mathbb{C}^n)\big)$ and Fano schemes}
\author{Andr\'as P. Juh\'asz}
\address{Alfr\'ed R\'enyi Institute of Mathematics Budapest, Hungary}
\email{juhasz.andris@gmail.com}
\author{L\'aszl\'o M. Feh\'er}
\address{Eotvos University Budapest}
\email{lfeher63@gmail.com}

\begin{document}
\begin{abstract}
  We prove that the $d$-dependence of $c(\Pol^d(\C^n))$, the Chern class of the
  $\GL(n)$-representation of degree $d$ homogeneous polynomials in $n$ complex variables
  is polynomial.
  We also study the asymptotics of the polynomial $d$-dependence of this Chern class.
  We apply these results to solve a conjecture of Manivel on the degree of varieties
  of hypersurfaces containing linear subspaces.
  We also prove new formulas for the degree and Euler characteristics of Fano schemes of lines
  for generic hypersurfaces.

  To prove the polynomial dependence of $c(\Pol^d(\C^n))$ we introduce the notion of Striling
  coefficients.
  This can be considered as a generalization of Stirling numbers. We express the Stirling
  coefficients in terms of specializations of monomial symmetric polynomials 
  so we can deduce their polynomiality and, in certain cases, their asymptotic behaviour.
\end{abstract}
\maketitle

\tableofcontents
\section{Introduction}
The original motivation for this paper was to calculate the Euler characteristics of various
coincident root loci of binary forms using equivariant Chern-Schwartz-MacPherson classes.

Those results will be published in a separate paper because we quickly realized that even the ``trivial" case,
the total Chern class of the $\GL(2)$-representation $\Pol^d(\C^2)$
\begin{equation}\label{eq:cpoldc2}
  c\big(\!\Pol^d(\C^2)\big)=\prod_{t=0}^{d}(1+tx_1+(d-t)x_2).
  \end{equation}
is worth studying:
It is a symmetric polynomial in the variables $x_1$ and $x_2$, and
our first goal was to study the coefficients of this polynomial in various bases of
the ring of symmetric polynomials in $x_1$ and $x_2$.

These coefficients are easy to calculate from \eqref{eq:cpoldc2} using a computer, but we show in this paper that these
coefficients are polynomials in $d$, and we calculate their leading terms.

Later we realized that our approach is suitable for a similar analysis of the coefficients of
the Chern class of the $\GL(n)$-representation $\Pol^d(\C^n)$,
\begin{equation}\label{eq_cPoldCn_singleproduct}
c\big(\!\Pol^d(\mathbb{C}^n)\big) = \prod_{\substack{(d_1,\dots,d_n) \\ d_1+\dots+d_n=d}}  \!\!
\left( 1+d_1 x_1+\dots+d_n x_n \right)
\end{equation}
for all $n$'s, not necessarily 2.

These results can be translated into enumerative result. For example, we prove a conjecture of Manivel on the degree of varieties of hypersurfaces containing linear subspaces (Theorem \ref{thm:degSigma_leadingterm}).

We will also prove new closed formulas for the degree and Euler characteristics of Fano schemes of lines for generic hypersurfaces: Theorem \ref{thm-deg-of-fano-of-lines}, and \eqref{fano-chi-delta1} and \eqref{fano-chi-delta2} in Section \ref{sec:euler}).

To study $c\big(\!\Pol^d(\C^2)\big)$---and more generally $c\big(\!\Pol^d(\C^n)\big)$---we develop a theory for products like \eqref{eq:cpoldc2}. We will call these products \emph{rising products}, see Definition \ref{def_risingprod_Stirlingcoef}.
A key result of the paper is that the coefficients of a rising product are polynomials in $d$. We will call these coefficients \emph{Stirling coefficients}.
The  Stirling numbers $\stir{m}{k}$, $ \begin{Bmatrix}
m \\
k 
\end{Bmatrix}  $ and the binomial coefficients $\displaystyle\binom{m}{k}$ are the simplest examples. Stirling coefficients include the generalized Stirling numbers defined in \cite{Tweedie}, central factorial numbers, and other generalizations. The literature on these generalizations is vast (See e.g. \cite{cigler,komatsu-stirling-level2,r-central,schmidt-gen-stirling}). Many of these generalizations are special cases of our Stirling coefficients.  For this reason we hope that the theory of rising products and their connection with the \emph{arithmetic specialization}  of monomial symmetric polynomials will be of independent interest. Here we  call $f(1,2,\dots,n)$ the arithmetic specialization of the polynomial $f\in \Q[x_1,x_2,\dots,x_n]$.
 
\subsection{Structure of the paper}
In Section \ref{sec_Stirlingcoefs} we introduce the notion of rising products. In Theorem \ref{prop:rising} we give a formula for the coefficients in terms of arithmetic specializations of monomial symmetric polynomials---what we will call Stirling coefficients---proving that they are polynomials in $d$. 
We also give degree estimates to these polynomials in Section \ref{sec:asymp-stirling}.

In Section \ref{sec_cpold} we apply the theory of rising products to prove polynomiality of the coefficients of $c\big(\!\Pol^d(\C^n)\big)$ in Theorem \ref{thrm_cpoldCnispoly} and calculate its leading term in \ref{thrm_leadingtermcPoldCn}. In Section \ref{sec:cpoldc2-rising} we give a closed formula for the coefficients for $n=2$. In Section \ref{subsec_23_345} we state Theorem \ref{thrm_23_345thrm} wich gives a significantly lower estimate for the coefficients in the elementary symmetric polynomial basis. We formulate a conjecture on the leading terms of these coefficients.
In Section \ref{sec_onclosedformulas} we contemplate on the role of closed formulas and generating functions and give a closed formula for the Euler class $c_{d+1}\big(\!\Pol^d(\C^2)\big)$ in Theorem \ref{thm-euler} which will be key to calculate the Euler characteristics of Fano schemes in Section \ref{sec:euler}. In Section \ref{sec:separately} we conjecture that $c_k\big(\!\Pol^d(\C^n)\big)$ is separately polynomial in $n$ and $d$.

In Section \ref{sec:enum} we turn to enumerative applications. In Section \ref{subsec_hypersurfswithlinsubspaces} we study the degree of varieties of hypersurfaces containing linear subspaces and prove a conjecture of Manivel in Theorem \ref{thm:degSigma_leadingterm}.
In Section \ref{sec:deg-of-fano} we give a closed formula for the
 degrees of Fano schemes of lines of generic hypersurfaces of degree $d$ in Theorem \ref{thm-deg-of-fano-of-lines}. In Section \ref{sec:euler}
we give closed formulas for the Euler characteristics of Fano schemes of lines of generic hypersurfaces of degree $d$.

In Appendix \ref{sec_23_345thrm}---via an analysis of the orbits of the $\GL(n)$-representation
$\Pol^d(\C^n)$ that leads to a quasi-polynomial behaviour---we prove Theorem
\ref{thrm_23_345thrm}.

\subsection*{Acknowledgements}
A.J. was partially supported by the NKKP grant K 146401.
We thank \'Arp\'ad T\'oth for an inspiring conversation on the rising products and Bal\'azs K\H om\H uves for sharing his knowledge with us.

\section{Rising products and Stirling coefficients}\label{sec_Stirlingcoefs}
To study \eqref{eq:cpoldc2} we introduce the notion of rising products. We give a formula for the coefficients---what we will call Stirling coefficients---proving that they are polynomials in $d$. We also give degree estimates to these polynomials in Section \ref{sec:asymp-stirling}.

\subsection{Notations}
The paper contains lots of lengthy expressions and technical arguments. To make life easier,
we try to collect some notions and notations we will use.

\medskip

Let us write $\lambda \vdash k$ for (proper integer) partitions of $k$.
By a \emph{weak partition} we mean a decreasing sequence of non-negative integers,
e.g. $\lambda=(5,3,3,0,0,0)$. We will also use the exponential form $(0^3,3^2,5)$.
For any (weak) partition $\lambda=\left(0^{m_0},1^{m_1},2^{m_2},\dots  \right)$ let
\[ \mult(\lambda)=\left( m_0,m_1,m_2,\dots \right) \]
denote the multiplicity vector of the elements of the partition.

Extending the notion of (proper integer) partitions,
we define a \emph{partition of a vector} $H \in \N^{\infty}$ to be a tuple
\begin{equation*}\label{eq_def_partitionofavector}
  J=\left( J_1, J_2, \dots\right) \vdash H,
\end{equation*}
such that
$0 \neq J_s \in \mathbb{N}^{\infty}$, $\sum J_s = H$ and $J_1 \geq J_2 \geq \dots$, where
$\geq$ denotes some (e.g. lexicographical) ordering of $\mathbb{N}^\infty$.
For example,
\begin{equation} \label{vectpart}
  \begin{split}
    \left\{ J \vdash \left( 2,0,1,0,\dots \right) \right\}= & \big\{ \! \left( (2,0,1,0,\dots) \right), \\
      & \left( (2,0,0,0,\dots),(0,0,1,0,\dots) \right),\\
      & \left( (1,0,1,0,\dots),(1,0,0,0,\dots) \right),\\
    & \left( (1,0,0,0,\dots),(1,0,0,0,\dots), (0,0,1,0,\dots) \right) \! \big\}.
  \end{split}
\end{equation}
To shorten the notation we will denote by $1_i \in \N^{\infty}$ the vector which is 1 at its $i$-th coordinate and 0 everywhere else.

We will also use the exponential notation, for example the last vector partition in \eqref{vectpart} can also be written as $(1_1^2,1_3)$.

For a vector $H=(H_1,\dots,H_n)$ the partition
\[ \left( 1_1^{H_1},\dots,1_n^{H_n} \right) \vdash H \]
will prove to be important. 

The notion of multiciplity naturally extends to multiciplities of partitions of vectors,
$\mult(J)$. For example, $\mult\!\left( \left( 1_1^{H_1},\dots,1_n^{H_n} \right) \right)=
(H_1,H_2,\dots,H_n)$.

We will denote the length of partitions $\lambda \vdash k$ and $J \vdash H$ by $l(\lambda)$
and $l(J)$. In both cases, length is finite by definition.

Given a vector $v$, $v_i$ $(i \in \N)$ will implicitely be used to denote its entries. 
For the vectors $v$ and $x$
let
\[ |v|=\sum_{i} v_i, \quad v!=\prod_{i} v_i! \quad \text{ and } x^v=\prod_{i} x_i^{v_i} .\]

\medskip 

Finally, the coefficient of a monomial $x^v$ in a polynomial $p$ will be denoted by
\[ \coef(x^v,p). \]

\subsection{Introducing rising products and Stirling coefficients}
\begin{definition}\label{def_risingprod_Stirlingcoef}
  Suppose that the formal power series
\[ P(d_0,\dots,d_r,t,x)=\sum_{E} P_E(d_0,\dots,d_r,t) x^{E} \in
\mathbb{Q}\left[ d_0,\dots,d_r,t \right]\left[ \left[ x \right] \right] \]
satisfies $P(d_0,\dots,d_r,t,0)=1$, where $x=(x_1,\dots,x_n)$.
Let $K(d_0,\dots,d_r) \in \mathbb{Q}[d_0,\dots,d_r]$ be an integer valued polynomial
in the ``parameters'' $d_0,\dots,d_r \in \N$. Then we call the product
\begin{equation}\label{eq_oneparameterproductstart}
  S[P,K](d_0,\dots,d_r,x):=\prod_{t=0}^{K(d_0,\dots,d_r)} P(d_0,\dots,d_r,t,x)=
  \sum_{H \in \N^n} S[P,K]_H(d_0,\dots,d_r) x^H
\end{equation}
a \emph{rising product}, and the coefficients $S[P,K]_H(d_0,\dots,d_r)$ \emph{Stirling coefficients}.
\end{definition}
We will often abbreviate by $d$ the sequence of parameters $d_0,\dots,d_r$.
We will also use the shorthand notations $S[P]:=S[P,K]$ and $S[P]_H:=S[P,K]_H$ in case 
$K(d)=d$, i.e. $K(d_0)=d_0$.

The goal of this section is to investigate the
$\left\{ d_0,\dots,d_r \right\}$-dependence of the Stirling coefficients.
We will show that the Stirling coefficients form polynomials. More precisely,
\begin{theorem}
\label{thrm_polynomialityofcoeffs_generalsetup}
  Suppose that the formal power series
  $ P(d_0,\dots,d_r,t,x) \in \mathbb{Q}[d_0,\dots,d_r,t][[x]] $ satisfies $P(d_0,\dots,d_r,t,0)=1$,
  where $x=(x_1,\dots,x_n)$ and $K(d_0,\dots,d_r) \in \mathbb{Q}[d_0,\dots,d_r]$ is integer valued.
  
   Let
  \begin{equation*}\label{eq_productofpolys}
    S[P,K](d_0,\dots,d_r,x):=
    \prod_{t=0}^{K(d_0,\dots,d_r)} P(d_0,\dots,d_r,t,x)=
    \sum_{H \in \N^n} S[P,K]_H(d_0,\dots,d_r) x^H.
  \end{equation*}

 Then for all multiindex $H$ the Stirling coefficient  
  $S[P,K]_H(d_0,\dots,d_r)$ can be described by a polynomial in  $\mathbb{Q}[d_0,\dots,d_r]$
  if $K(d_0,\dots,d_r) \ge -1$.
\end{theorem}
In other words, there exists a formal power series in $\Q[d_0,\dots,d_r][[x]]$ that evaluates to
the rising product $S[P,K](d_0,\dots,d_r,x)$ for every $d_0,\dots,d_r$ with
$K(d_0,\dots,d_r) \ge -1$.
\bigskip

Our strategy to prove Theorem \ref{thrm_polynomialityofcoeffs_generalsetup} is that we first
prove a more general statement (Proposition \ref{prop_productformula}) where we ``replace $t$
with a variable $y_t$''. This allows us to show a deep connection with the monomial symmetric
polynomials. Finally in Lemma \ref{lemma_tildeMlambdaspolys} we show that arithmetic
specialization (i.e. substituting $t$ into the variable $y_t$) of a monomial symmetric polynomial
is a polynomial.

\subsection{First examples of rising products}  \label{sec:firstex}
In the examples below $K(d)=d$:
\begin{example}\label{ex_i}
  For $P(d,t,x)=1+tx$  we have
  \[ S[P](d,x)=\prod_{t=0}^{d}(1+tx)=\sum_{h=0}^{d+1}S[P]_h(d)x^h\]
  and
  \begin{equation*}
    S[P]_h(d)=\sigma_h(1,2,\dots,d)=\stir{d+1}{d+1-h},
\end{equation*}
where $\sigma_h$ is the $h$-th elementary symmetric polynomial and $\stir a b$ denotes the
Stirling number of the first kind.

It is well known that $\stir{d+1}{d+1-h}$ is a polynomial of degree $2h$ in $d$:
\begin{equation*}\label{stirling2binomial}
  S[P]_h(d)=\stir{d+1}{d+1-h}=\sum_{j=1}^{h}E_2(h,j)\binom{d+j}{2h},
\end{equation*}
where the $E_2(h,j)$'s are the  second-order Eulerian numbers.
For example,
\[ S[P]_1(d)=\binom{d+1}{2}, \,
  S[P]_2(d)=\binom{d+1}{4}+2\cdot\binom{d+2}{4}, \,
S[P]_3(d)=\binom{d+1}{6}+8\cdot\binom{d+2}{6}+6\cdot\binom{d+3}{6} . \]
\end{example}

\begin{example}\label{ex_secondkind}
 For $P(d,t,x)=\frac{1}{1-tx}=1+tx+t^2x^2+\cdots$  we have
\[ S[P](d,x)=\prod_{t=0}^{d}\frac{1}{1-tx}=\sum_{a=0}^{\infty}S[P]_a(d)x^a\]
and
\begin{equation*}
S[P]_a(d)=h_a(1,2,\dots,d)=\begin{Bmatrix}
d+a \\
a 
\end{Bmatrix},
\end{equation*}
where $h_a$ is the $a$-th complete homogeneous symmetric polynomial and $\begin{Bmatrix}
d+a \\
a 
\end{Bmatrix}$ denotes the corresponding
Stirling number of the second kind.
\end{example}

\begin{example}\label{ex_i^s}
 For $P(d,t,x)=1+t^sx$  we have
 \[ S[P](d,x)=\prod_{t=0}^{d}(1+t^s x)=\sum_{h=0}^{d+1}S[P]_h(d)x^h, \]
 where
 \begin{equation*}
   S[P]_h(d)=\sigma_h(1,2^s,\dots,d^s)=m_{s^h}(1,2,\dots,d),
 \end{equation*}
 where $m_\lambda$ is the monomial symmetric polynomial corresponding to the partition $\lambda$.
\end{example}
\begin{remark}
These numbers were called generalized Stirling numbers by Tweedie in \cite{Tweedie} and Stirling
numbers of the first kind with level $s$ by Komatsu in \cite{komatsu-stirling-level2}.
\end{remark}

\begin{example} \label{ex_i+i^2}
  For $P(d,t,x_1,x_2)=1+t x_1 + t^2x_2$ we have
  \[ S[P](d,x_1,x_2)=\prod_{t=0}^{d}(1+tx_1+t^2x_2)=
  \sum_{H \in \N^2}S[P]_{H}(d) x_1^{H_1} x_2^{H_2},  \]
  and with a similar argument we get that
  \[ S[P]_{H}(d)=m_{1^{H_1}2^{H_2}}(1,2,\dots,d).\]
\end{example}

These examples suggests that the coefficients are built from arithmetic specializations of monomial symmetric polynomials. The simplest example seems to be an exception:
\begin{example} \label{ex_i^0}    For $P(d,t,x)=1+x$
  \begin{equation*}\label{binom}
    \prod_{t=0}^d (1+x)=(1+x)^{d+1}=\sum_{h=0}^{d+1}\binom{d+1}{h}x^h.
  \end{equation*}
  The binomial coefficients are not arithmetic specializations of monomial symmetric polynomials.
  However, there is a natural extension of monomial symmetric polynomials which solves this problem. We introduce this extension in the next section.
\end{example}

\subsection{Versions of monomial symmetric polynomials}
To understand the connection between rising products and monomial symmetric polynomials we look at their generating functions.
To keep our formulas shorter, let us abbreviate by $\underline{y}_v$ the list of variables $y_0,y_1,\dots,y_v$.
\begin{proposition}\label{prop_genfn_monomsymproper}
  For any partition $\lambda=\left( 1^{m_1},\dots,n^{m_n} \right)$ we have
  \begin{equation*}\label{eq_genfn4monsymmpolys}
   m_\lambda(\underline{y}_v)= \coef\left( x_1^{m_1}\dots x_n^{m_n},
    \prod_{t=0}^v \Big( 1+ \sum_{s=1}^n y_t^{s} x_s \Big) \right)
.
  \end{equation*}
\end{proposition}

Then Examples \ref{ex_i}, \ref{ex_i^s} and \ref{ex_i+i^2} easily follow by substituting
$t$ into the variable $y_t$ for $t=0,\dots,v$.
We will call this substitution the arithmetic specialization.

\medskip

Modifying the above generating function, we can extend the notion of monomial symmetric polynomials
from partitions to \emph{weak partitions} possibly containing zeros.
\begin{definition}\label{def_monomsymmpolyweak}
  Let $\lambda=\left( 0^{m_0},1^{m_1},\dots,n^{m_n} \right)$ be a weak partition. Then the monomial symmetric function corresponding to $\lambda$ is defined by 
  \begin{equation}\label{eq_genfn4monsymmpolys_weakpartitions}
    m_\lambda(\underline{y}_v):=
    \coef\left( x_0^{m_0} x_1^{m_1}\dots x_n^{m_n},
    \prod_{t=0}^v \Big( 1+ \sum_{s=0}^n y_t^{s} x_s \Big) \right)
  \end{equation}
\end{definition}
Note that for every weak partition $\lambda=\left( 0^{m_0},1^{m_1},\dots,n^{m_n} \right)$ and
$v \in \N$ we have
\begin{equation}\label{eq_monsymmpolforweak_vsfornozeropart}
  m_\lambda(\underline{y}_v)=\binom{v+1-(m_1+\dots+m_n)}{m_0} m_{\lambda^*}(\underline{y}_v),
\end{equation}
where $\lambda^*$ denotes the nonzero part of $\lambda$. In particular, for
$\lambda=\left( 0^{m_0} \right)$ we get back the
polynomial
$m_{0^{m_0}}(\underline{y}_v)=\binom{v+1}{m_0}$ of Example \ref{ex_i^0}.

\medskip

Augmented monomial symmetric polynomials corresponding to partitions
$\lambda=\left( 1^{m_1},\dots,n^{m_n} \right)$ are usually defined as
\begin{equation}\label{eq_augmentedmonsymm_def}
  \tilde{m}_\lambda(\underline{y}_v)=\mult(\lambda)! m_\lambda(\underline{y}_v).
\end{equation}
It has an other description,
\begin{equation}\label{eq_augmentedmonsymm_injectivefns}
  \tilde{m}_\lambda(\underline{y}_v)=
  \sum_{f\in\rm{Inj} }
  \prod_{s=1}^{l(\lambda)} y_{f(s)}^{\lambda_s},
\end{equation}
---where the summation is over the set $\rm{Inj}$ of injective functions
 from
$\left\{ 1,\dots,l(\lambda) \right\}$ to
$\left\{ 0,1,\dots,v \right\}$---which is better-suited for our needs,
see Proposition \ref{prop_productformula}.
We also use \eqref{eq_augmentedmonsymm_injectivefns} to define augmented monomial symmetric
polynomials for weak partitions $\lambda=\left( 0^{m_0},1^{m_1},\dots,n^{m_n} \right)$.
Then \eqref{eq_augmentedmonsymm_def} still holds but
 the analog of \eqref{eq_monsymmpolforweak_vsfornozeropart} becomes
 \begin{equation}\label{eq_augmonsymmpolforweak_vsfornozeropart}
  \tilde{m}_\lambda(\underline{y}_v)=
  \binom{v+1-(m_1+\dots+m_n)}{m_0} m_0! \,  \tilde{m}_{\lambda^*}(\underline{y}_v)
\end{equation}

\medskip

Finally, we will also write $m_w(\underline{y}_v)$ and $\tilde{m}_w(\underline{y}_v)$ for vectors
$w$ of non-negative entries, not necessarily decreasing.
By this we mean the (augmented) monomial symmetric polynomial indexed by the (weak) partition we
get by sorting $w$.

\subsection{The parametrized rising product formula }
We start making steps towards the proof of Proposition
\ref{thrm_polynomialityofcoeffs_generalsetup} by
generalizing \eqref{eq_genfn4monsymmpolys_weakpartitions}.
\begin{proposition}\label{prop_productformula}
  Suppose that the formal power series
  $ P(d,y,x) \in \mathbb{Q}[d,y][[x]] $
  satisfies $P(d,y,\underline{0})=1$,  where
  $d=(d_0,\dots,d_r)$.
  
  Consider the expansion of the following product
  \begin{equation*}
    A[P](d,\underline{y}_v,x):=
    \prod_{t=0}^{v} P(d,y_t,x)=
    \sum_{H} A[P]_H(d,\underline{y}_v) x^H.
  \end{equation*}
  Then for each coefficient we have
  \begin{equation}\label{eq_productdformulacoefficient}
  A[P]_H(d,\underline{y}_v)=
    \sum_{J\vdash H}
    \frac{1}{\mult(J)!}
    \sum_{\lambda \in \N^{l(J)} }
    \left( \tilde{m}_\lambda(\underline{y}_v)
\prod_{s=1}^{l(J)} P_{J_s,\lambda_s}(d) \right),
\end{equation}
  where 
 \begin{equation}\label{defofpjm}
  P(d,y,x)=\sum_E P_{E}(d,y) x^E =
  \sum_E \sum_{m } \left( P_{E,m}(d) y^m \right) x^E.
 \end{equation}
 Notice that in \eqref{defofpjm} there are only finitely many $m$'s for which $P_{E,m}(d)$ is
 nonzero.
 This implies that there are only finitely many $\lambda$'s in
 \eqref{eq_productdformulacoefficient} for which the corresponding term is nonzero,
 so the sum is well defined.
 More explicitely, if we denote by $I_P(E)$ the support of $P_E(d)$, then it is enough to run the sum for $\lambda \in \bigtimes_{s=1}^{l(J)} I_P(J_s)$.
 \end{proposition}

 \begin{proof}
   Multiplying the terms, we can easily expand the coefficient
   $A[P]_H(d,\underline{y}_v)$ as
   \begin{equation*}\label{eq_coeff_combinatorialformula}
     A[P]_H(d,\underline{y}_v)=
     \sum_{J \vdash H} \underbrace{\frac{1}{\mult(J)!}
     \sum_{f \in \rm{Inj}}
   \prod_{s=1}^{l(J)} P_{J_s}(d,y_{f(s)})}_
     {A_{J}(d,\underline{y}_v)},
   \end{equation*}
   where the latter summation is over the set $\rm{Inj}$ of injective functions from
$\left\{ 1,\dots,l(J) \right\}$ to $\left\{ 0,1,\dots,v \right\}$.
   
   For each $J=\left( J_1,\dots,J_{l(J)} \right) \vdash H$ we have
   \begin{equation*}  \label{eq_coeff_combinatorialformulafinal}
   \begin{split}
   A_{J}(d,\underline{y}_v)&=
   \frac{1}{\mult(J)!}
   \sum_{f \in \rm{Inj}}
   \prod_{s=1}^{l(J)} P_{J_s}(d,y_{f(s)})
   \\
   &=\frac{1}{\mult(J)!} \sum_{f \in \rm{Inj}}
   \prod_{s=1}^{l(J)} \sum_{m } P_{J_s,m}(d) y_{f(s)}^m 
   \\
   &=\frac{1}{\mult(J)!} \sum_{f \in \rm{Inj}}
   \sum_{\lambda \in \N^{l(J)}}
   \prod_{s=1}^{l(J)} P_{J_s,\lambda_s}(d) y_{f(s)}^{\lambda_s}
   \\
   &=\frac{1}{\mult(J)!}
 \sum_{\lambda \in \N^{l(J)}}
 \prod_{s=1}^{l(J)} \left( P_{J_s,\lambda_s}(d) \right)
 \sum_{f \in \rm{Inj}} \prod_{s=1}^{l(J)}
 y_{f(s)}^{\lambda_s}
 \\
  &=\frac{1}{\mult(J)!}
  \sum_{\lambda \in \N^{l(J)}}
  \left( \tilde m_\lambda(\underline{y}_v)
  \prod_{s=1}^{l(J)}P_{J_s,\lambda_s}(d) \right),
   \end{split}
\end{equation*}
where---as we mentioned before at \eqref{eq_productdformulacoefficient}---there are only finitely many $\lambda$'s  for which the corresponding term is nonzero, so the sum is well defined.

 \end{proof}

 If the formal power series $P(d,y,x) \in \Q[d,y][[x]]$ is complicated,
 then the formula \eqref{eq_productdformulacoefficient} provided by Proposition
 \ref{prop_productformula} will also be complicated, which limits its usefullness for actual
 calculations.

 There are some simple cases, however, where we can obtain reasonable formulas. For example, 
 generalizing Proposition \ref{prop_genfn_monomsymproper} and Definition
 \ref{def_monomsymmpolyweak},

\begin{proposition}\label{prop:simple-coeffs}
Let $E=\left( E_1,\dots,E_{n}\right)=\left( 0^{e_0},1^{e_1},\dots,z^{e_z} \right) \in \N^{n}$
and consider the polynomial
\[ P(d,y,x)= 1+ \sum_{s=1}^n y^{E_s}x_s. \]
Then for any exponent vector $H=(H_1,\dots,H_n)$ with coordinates $H_s \neq 0$ we have
\[
  A[P]_H(d,\underline{y}_v)=
  \left( \left. H \right|_{E_i=1} \right) \dots \left( \left. H \right|_{E_i=z} \right)
      m_\lambda(\underline{y}_v),
    \]
where
\begin{gather*}
  \text{for elements $j$ of $E$} \quad \left. H \right|_{E_i=j}:=\left( H_i : E_i=j \right)
    \text{ is a vector},
    \\
    \text{for vectors }V=(V_1,\dots,V_t) \quad \left( V \right):=\binom{|V|}{V_1,\dots,V_t} \text{ is the multinomial coefficient and}
  \\
\lambda:=\left(
  0^{\left| \left. H \right|_{E_i=0}\right|},
    1^{\left| \left. H \right|_{E_i=1} \right|}, \dots,z^{\left| \left. H \right|_{E_i=z} \right|}
    \right) \text{ is the corresponding weak partition.}
\end{gather*}
\end{proposition}

\begin{example}
To determine the coefficient of
\[ x_1^1x_2^2x_3^1x_4^3x_5^1x_6^1x_7^2x_8^5 \]
in
\[ A[P](d,\underline{y}_v,x)=\prod_{t=0}^v \left( 1+ y_t^0 x_1+ y_t^0 x_2+ y_t^1 x_3+ y_t^1 x_4+ y_t^3 x_5+ y_t^3 x_6+
y_t^3 x_7+ y_t^4 x_8 \right) \]
let
\[ H=\left( 1,2,1,3,1,1,2,5 \right) \quad \text{ and } \quad E=\left( 0,0,1,1,3,3,3,4 \right) .\]
Then
\begin{gather*}
  \left. H \right|_{E_i=0}=\left( 1,2 \right) \quad
  \left. H \right|_{E_i=1}=\left( 1,3 \right) \quad
    \left. H \right|_{E_i=3}=\left( 1,1,2 \right) \quad
      \left. H \right|_{E_i=4}=\left( 5 \right), \\
	\lambda=(0^3,1^4,4^4,4^5),
\end{gather*}
and
\[ A[P]_H(d,\underline{y}_v) = \coef(x^H, A[P](d,\underline{y}_v,x))=
\binom{4}{1,3}\binom{4}{1,1,2}\binom{5}{5} m_{\lambda}(\underline{y}_v). \]
\end{example}

\subsection{Expressing Stirling coefficients in terms of arithmetic specialization of monomial symmetric polynomials}
By definition, the rising product \eqref{eq_oneparameterproductstart} is obtained from its parameterized version
$A[P](d,\underline{y}_v,x)$ by substitution:
\begin{equation*}\label{eq:main-sub}
  S[P,K](d,x)=\sub\big[A[P](d,\underline{y}_v,x);
  v \mapsto K(d),y_t\mapsto t\big].
\end{equation*}
Therefore, its coefficients can also be obtained by substitution.
\begin{definition}\label{def:arithmetic-m} We denote the arithmetic specialization of the augmented monomial symmetric polynomial by
\[ \tilde{M}_\lambda(v):=\sub\left[ \tilde{m}_\lambda(\underline{y}_v);y_t \mapsto t \right]. \]
\end{definition}

Using this notation, Proposition \ref{prop_productformula} gives that

\begin{theorem}\label{prop:rising} Given a rising product $S[P,K](d,x)$,
for every $H \in \N^{\infty}$ the corresponding Stirling coefficient is
\begin{equation}\label{eq_Stirlingcoeffs_inMtildas}
  S[P,K]_H(d_0,\dots,d_r)=
  \sum_{J\vdash H}
  \frac{1}{\mult(J)!}
  \sum_{\lambda \in   \N^{l(J)} }
  \prod_{s=1}^{l(J)} \left( P_{J_s,\lambda_s}(d) \right)
  \underbrace{\sub\left[ \tilde{m}_\lambda(\underline{y}_v);v \mapsto K(d),y_t \mapsto t \right]}_
  {\tilde{M}_\lambda(K(d))}.
\end{equation}
\end{theorem}

\medskip

Therefore, to complete the proof of the polynomiality of the Stirling coefficients
all is left to do is to show
that for every weak partition $\lambda$ the function  $v \mapsto \tilde{M}_\lambda(v)$
is in $\Q[v]$. More precisely, 

\begin{lemma}\label{lemma_tildeMlambdaspolys}
  For every weak partition $\lambda$ there exists a polynomial $\tilde{M}_\lambda(v) \in \Q[v]$
  such that
  \[ \tilde{M}_\lambda(v)=\tilde{m}_{\lambda}(0,1,\dots,v)
  \quad (\text{ for } v \ge -1
    \footnote{By definition, $\tilde{M}_{\lambda}(-1)=0$.}
). \]
  Moreover, the leading term of $\tilde{M}_\lambda(v)$ is 
 \begin{equation}\label{eq_leadingtermMtuple}
   \prod_i \left( \frac{1}{\lambda_i+1} \right) v^{|\lambda|+l(\lambda)}
  =  \prod_i  \frac{v^{\lambda_i+1}}{\lambda_i+1}.
\end{equation}
\end{lemma}

\begin{proof}
For every partition $\lambda=(1^{m_1},\dots,n^{m_n})$ the augmented monomial symmetric polynomial
$\tilde m_\lambda$ can be expressed as a polynomial of power sum symmetric polynomials $p_k$
(see \cite{macdonald1998symmetric}):
\begin{equation}\label{eq_augmonsymm_inpowersumbasis}
  \tilde m_{\lambda}(\underline{y}_v)=p_\lambda(\underline{y}_v)+
  \sum_{\mu \in \partial \lambda} c_{\mu} p_{\mu}(\underline{y}_v)
\end{equation}
for some $c_\mu \in \mathbb{Z}$, where $\partial \lambda$ denotes the set of partitions coming from $\lambda$ with at least two of its elements merged and $p_{\eta}=p_{\eta_1}\dots p_{\eta_l}$ for $\eta=\left( \eta_1,\dots,\eta_l \right)$.

If we define $\tilde{M}_\lambda(v)$ to be the 
$y_t \mapsto t$ substitute of the right-hand side of \eqref{eq_augmonsymm_inpowersumbasis},
then e.g. Faulhaber's formula for
\begin{equation*}
p_q(\underline{y}_v)|_{y_t \mapsto t}=\sum_{k=0}^v k^q
\end{equation*}
immediately implies that
$\tilde{M}_\lambda(v) \in \Q[v]$.

The leading terms of the $\tilde{M}_\lambda(v)$'s can also be obtained from \eqref{eq_augmonsymm_inpowersumbasis}: as elements of $\partial \lambda$ have length smaller than that of the partition $\lambda$, we can deduce from Faulhaber's formula that the leading term of $\tilde{M}_\lambda(v)$
is the leading term of the $\lambda$-summand, that is \eqref{eq_leadingtermMtuple}. 

For weak partitions $\lambda=\left( 0^{m_0},1^{m_1},\dots,n^{m_n} \right)$ we can use \eqref{eq_augmonsymmpolforweak_vsfornozeropart} to show that $\tilde{M}_\lambda(v)$
is also in $\Q[v]$ with leading term as in \eqref{eq_leadingtermMtuple}.

Note that the polynomial $\tilde{M}_\lambda(v)$ will work for every $v \ge -1$:
  $\tilde{M}_\lambda(-1)=0$ for every weak partition $\lambda$ as the Faulhaber's formulas
  (or the binomial coefficient if $\lambda^*$ is empty) are divisible by $v+1$.
\end{proof}

\begin{remark}\label{rmrk_Mlambdaworksfornonnegative}
  For every weak partition $\lambda$, by definition, $\tilde{m}_\lambda(\underline{y})=0$
  for every $0 \le v < l(\lambda)-1$. For proper partitions
  $\tilde{m}_\lambda(\underline{y}_{l(\lambda)-1})=0$ also holds.
  This means that the arithmetic specialization $\tilde{M}_{\lambda}(v) \in \Q[v]$ is divisible by
\[ (v+1)v\dots(v-(l(\lambda)-2)) \quad \text{ or } \quad  (v+1)v\dots(v-(l(\lambda)-1)) \text{ if } \lambda \text{ is not weak.} \]
For example, 
\[ \tilde{M}_{(2,0,0)}(v)=\frac{1}{6}(v+1)v^2(v-1)(2v+1). \]
\end{remark}

Lemma \ref{lemma_tildeMlambdaspolys} explains the hypothesis $K(d_0,\dots,d_r) \ge -1$ in Theorem
\ref{thrm_polynomialityofcoeffs_generalsetup}: for such $(d_0,\dots,d_r)$'s the
function $\tilde{M}_\lambda(K(d_0,\dots,d_r))$ is described by a polynomial in
$\Q[d_0,\dots,d_r]$.

By \eqref{eq_Stirlingcoeffs_inMtildas}, 
this finishes the proof of Theorem \ref{thrm_polynomialityofcoeffs_generalsetup}.

\subsection{Asymptotic behaviour of Stirling coefficients} \label{sec:asymp-stirling}
The coefficients of $c\big(\!\Pol^d(\C^n)\big)$ are Stirling coefficients that come from rising products 
where further assumption can be made. These assumptions lead to a simple relationship between the
(total) degrees of
$P_E(d,t)$ and $S[P,K]_H(d)$ of Definition \ref{def_risingprod_Stirlingcoef}.
The a priori upper bound for the $d$-degree of $S[P,K]_H(d)$ we can get this way is really useful,
because it allows us to use interpolation to calculate them.

\begin{proposition}\label{prop_stepLW}
  Suppose that the formal power series
  \[ P(d_0,\dots,d_r,t,x)=\sum_{E} P_E(d_0,\dots,d_r,t) x^E \in \mathbb{Q}[d_0,\dots,d_r,t][[x]] \]
  satisfies $P(d_0,\dots,d_r,t,0)=1$, where $x=(x_1,x_2,\dots)$ and let
  $L(d_0,\dots,d_r) \in \Z[d_0,\dots,d_r]$ be a linear polynomial.
  Recall that the corresponding rising product is defined as
  \begin{equation*}
    S[P,L](d_0,\dots,d_r,x):=\prod_{t=0}^{L(d_0,\dots,d_r)} P(d_0,\dots,d_r,t,x)=
    \sum_{H} S[P,L]_H(d_0,\dots,d_r) x^H,
  \end{equation*}
  and that the Stirling coefficients $S[P,L]_H(d_0,\dots,d_r)$ are in $\Q[d_0,\dots,d_r]$.

  Assume that there is a linear form $W: \mathbb{Z}^\infty \to \mathbb{Z}$ such that for every
  $E \in \mathbb{N}^\infty$, $\deg(P_E(d,t)) \leq W(E)$. Then
  \begin{enumerate}[label=\roman*)]
  \item \label{item_propstep1} $\deg(S[P,L]_H(d)) \leq W(H)+|H|$,
  \item \label{item_propstep2} if $\deg(S[P,L]_H(d))$ reaches $W(H)+|H|$
    (in which case the upper bound $\deg(P_E(d,t)) \leq W(E)$ is sharp
      \footnote{
	More precisely,	\begin{equation*}
	  \begin{gathered}
	    \deg(S[P,L]_H(d))=W(H)+|H| \\
	    \text{ for every } H
	  \end{gathered}
	  \Longleftrightarrow
	  \begin{gathered}
	    \deg(P_{1_i}(d,t))=W(1_i) \text{ and } \\
	    \sum_{\substack{m: \\ \deg(P_{1_i,m}(d))=W(1_i)-m}}
	    \left( \frac{1}{1+m}
	    \coef\left(d^{W(1_i)-m}, P_{1_i,m}(d)\right) \right) \neq 0.
	  \end{gathered}
	\end{equation*}
	In this case, the leading coefficient of $S[P,L]_H(d)$ can also be expressed, see
	Proposition \ref{prop_LWleadingcoeff}.
    }),
    then the leading term of $S[P,L]_H(d)$ comes from the
    \[ \prod_{t=0}^{L(d_0,\dots,d_r)}
  \left( 1+\sum_{i \geq 1} P_{1_i}(d_0,\dots,d_r,t) x_i\right) \]
  summand of $S[P,L](d,x)$.
\end{enumerate}
\end{proposition}

\begin{proof}
  By \eqref{eq_Stirlingcoeffs_inMtildas},
  \begin{equation*}
  S[P,L]_H(d_0,\dots,d_r)=
  \sum_{J\vdash H}
  \frac{1}{\mult(J)!}
  \sum_{\lambda \in \bigtimes_{s=1}^{l(J)} I_P(J_s)}
  \prod_{s=1}^{l(J)} \left( P_{J_s,\lambda_s}(d) \right)
  \tilde{M}_\lambda(L(d)),
\end{equation*}
therefore
\[ \deg(S[P,L]_H(d)) \leq
  \max_{J \vdash H}
  \max_{\lambda \in \bigtimes_{s=1}^{l(J)} I_P(J_s)}
  \deg\left( \prod_{s=1}^{l(J)} \left( P_{J_s,\lambda_s}(d) \right)
\tilde{M}_\lambda(L(d)) \right). \]

For each $\lambda_s \in I_P(J_s) $, by the hypothesis of the proposition,
$\deg\left( P_{J_s,\lambda_s}(d) \right) \leq W(J_s)-\lambda_s$.
Also, the linearity of $L(d_0,\dots,d_r)$ means that the degree part of
\eqref{eq_leadingtermMtuple} holds for $\tilde{M}_\lambda(L(d))$.
This gives that
\begin{multline}\label{eq_deglambdacontrib_estimate}
  \deg\left( \prod_{s=1}^{l(J)} \left( P_{J_s,\lambda_s}(d) \right)
  \tilde{M}_\lambda(L(d)) \right)=
  \sum_{s=1}^{l(J)} \left( \deg\left( P_{J_s,\lambda_s}(d) \right) \right)
  +|\lambda|+l(\lambda) \leq
  \\
  \sum_{s=1}^{l(J)}
  W(J_s)-\lambda_s+1+\lambda_s=W(H)+l(J) \leq W(H)+|H|,
\end{multline}
showing \ref{item_propstep1}.

$H=\left( H_1,\dots,H_n \right)$ has a single partition
\[ \left( 1_1^{H_1},\dots,1_{n}^{H_n} \right) \vdash H \]
that has length $|H|$ and where \eqref{eq_deglambdacontrib_estimate}
reach $|H|+W(H)$.
This proves \ref{item_propstep2}.
\end{proof}
\bigskip

\section{The $d$-dependence of \texorpdfstring{$c\big(\!\Pol^d(\C^n)\big)$}{c(Pold(Cn))}.}\label{sec_cpold}
  In this section we will apply the results of Section \ref{sec_Stirlingcoefs} to
  \begin{equation} \label{non-rising}
  c\big(\!\Pol^d(\C^n)\big)= \prod_{\substack{(d_1,\dots,d_n) \in \N^n \\ d_1+\dots+d_n=d}}
  \left( 1+d_1 x_1+\dots+d_n x_n \right), 
\end{equation}
  the Chern class of the $\GL(n)$-representation $\Pol^d(\C^n)$.
  This is not a rising product per se (except for the $n=2$ case),
  but can be considered as an iterated rising product as in \eqref{eq_cpoldcn_asseriesofproducts}.
  
  We will think of $n$ as a fixed parameter, and regard the class $c\big(\!\Pol^d(\C^n)\big)$ as
  a symmetric polynomial
  in the variables $x_1,\dots,x_n$.
  Our goal is to investigate the $d$-dependence of $c\big(\!\Pol^d(\C^n)\big)$.
  Using Theorem \ref{thrm_polynomialityofcoeffs_generalsetup}, we will show that
  \begin{theorem}\label{thrm_cpoldCnispoly}
    For every positive integer $n$ 
    \[ c\big(\!\Pol^d(\C^n)\big) \in \Q[d]\left[ \left[ x_1,\dots,x_n \right] \right]^{S_n}. \]
  \end{theorem}

  Theorem \ref{thrm_cpoldCnispoly}, as well as the other statements of this section, has various
  equivalent forms depending  on which basis we use for the ring of symmetric formal power
  series  $\Q\left[ \left[ x_1,\dots,x_n \right] \right]^{S_n}$.
  Among the several well-known bases we will consider those of the monomial symmetric polynomials
  $m_{\mu}(x_1,\dots,x_n)$, the
  Schur polynomials $s_\lambda(x_1,\dots,x_n)$ and the elementary symmetric polynomials
  $e_\nu(x_1,\dots,x_n)=\prod e_{\nu_i}(x_1,\dots,x_n)$.
  For example,
  Theorem \ref{thrm_cpoldCnispoly} can be expressed as 
  \begin{equation*}\label{eq_cpold_elemsymm_polynomial}
   c\big(\!\Pol^d(\C^n)\big) \in  \Q[d]\left[\left[e_1,\dots,e_n\right]\right]. 
  \end{equation*}
  Note that elements of the first two bases are indexed by partitions $\mu,\lambda$ of length
  at most $n$, while nonzero elementary symmetric polynomials are indexed by partitions $\nu$ 
  with their coordinates at most $n$.

  Writing $c_k(\Pol^d(\C^n))$ for the $k$-th Chern class of $\Pol^d(\C^n)$, i.e. the 
  homogeneous component of $\left\{ x_1,\dots,x_n \right\}$-degree $k$ of
  $c(\Pol^{d}(\mathbb{C}^n))=\sum_{k=0}^{d+1}c_k(\Pol^{d}(\mathbb{C}^n))$, Theorem
  \ref{thrm_cpoldCnispoly} can be reformulated as having
  \[  c_k(\Pol^{d}(\mathbb{C}^n)) \in \Q[d][x_1,\dots,x_n]^{S_n} \quad \text{ for every } k \in
  \N. \]
  
  \medskip 

  Following the proof of the polynomiality of the coefficients of $c(\Pol^{d}(\mathbb{C}^n))$,
  we turn to their asymptotic analysis. The resulting theorem can be stated the
  shortest as
  \begin{theorem}\label{thrm_leadingtermcPoldCn} The leading term of $c_k(\Pol^{d}(\mathbb{C}^n))\in\Q[e_1,\dots,e_n][d]$ is
    \[ \frac{1}{k!} \left( \frac{1}{n!} \right)^k e_1^k d^{nk}. \]
  \end{theorem}

  This theorem too has different equivalent forms depending on the chosen basis.
  For the enumerative applications in Section  \ref{sec:enum} we need the Schur polynomial basis,
 so let us highlight the corresponding versions of Theorem \ref{thrm_cpoldCnispoly} and 
 Theorem \ref{thrm_leadingtermcPoldCn}.
 \begin{corollary}\label{cor_cPoldCn_Schurcoeffs}
   For every partition $\lambda=\left(\lambda_1,\dots,\lambda_l  \right)$ the coefficients
   $b_\lambda(d)$ in
   \[ c\big(\!\Pol^d(\C^n)\big)=\sum_{\lambda} b_\lambda(d) s_\lambda(x_1,\dots,x_n) \]
   form a polynomial $b_\lambda(d) \in \Q[d]$ whose leading term is
   
   \[ \frac{\prod\limits_{1 \leq i < j \leq l} \left( \lambda_i - \lambda_j +j-i\right)}
     {(\lambda_1+l-1)! (\lambda_2+l-2)!\dots\lambda_l!}
     \left( \frac{1}{n!} \right)^{|\lambda|} d^{n|\lambda|}.
   \]
 \end{corollary}

 We collected the proofs of the above statements in Section \ref{subsec_proofscPoldCn}.

 \medskip 

 \begin{remark} \label{rmrk_c3ofpoldc2}
   Using our a priori knowledge of the $d$-degrees of the coefficients of 
   $c(\Pol^d(\C^n))$ in e.g. the Schur polynomial basis, we can interpolate the coefficients of
   $c(\Pol^d(\C^n))$. For example, we get that
   \begin{gather*}
     c_{3}(\Pol^d(\C^2))=
     \frac{1}{48}\,{d}^{2} ( d-1 )  ( d-2 )  ( d+1
     ) ^{2}s_3 +
     \frac{1}{24}\,{d}^{2} ( d-1 )  ( d+1 )  ( {d}^{2}+d
     +2 )s_{2,1}
     \\
     \text{or}
     \\
     \begin{aligned}
     c_3(\Pol^d(\C^3)))=&
     \frac{1}{6480}\,d(d-1)(d+3)(d+2)(d+1)(5d^4+20d^3-5d^2-50d-12) s_3+
     \\
     &\frac{1}{1296}\,d(d-1)(d+3)(d+2)(d+1)(d^2+2)(2d^2+8d+3) s_{2,1}+
     \\
     &\frac{1}{1296}\,d(d+3)(d+2)(d+1)(d^2+2)(d^3+3d^2+2d+12) s_{1,1,1}.
   \end{aligned}
 \end{gather*}
 \end{remark}

 \begin{remark}\label{rmrk_polynomialityfrom0}
  The proof of Theorem \ref{thrm_cpoldCnispoly} shows that the polynomials
  $c_k(\Pol^d(\C^n)) \in \Q[d][x_1,\dots,x_n]^{S_n}$ work
  for every $d \ge -1$
  \footnote{By definition, $c(\Pol^{-1}(\C^n))=1$.}.
  This means that we can
  use the data points for $d=-1,0,\dots$ when interpolating the coefficients of $c(\Pol^d(\C^n))$.

  For $d > 0$ the $\{x_1,\dots,x_n\}$-degree of $c\big(\!\Pol^d(\C^n)\big)$ is
  $\binom{d+n-1}{n-1}$, therefore we also have that for $k \neq 0$
  \[ c_k(\Pol^d(\C^n))=0 \text{ for every } d \ge -1 \text{ such that } \binom{d+n-1}{n-1}<k. \]
  In other words, $c_k(\Pol^d(\C^n)) \in \Q[d][x_1,\dots,x_n]^{S_n}$ is divisible by
  \[ (d+1) d (d-1)\dots(d-(d_0(k)-1)), \]
  where $d_0(k) \in \N$ is the smallest integer such that $\binom{d_0(k)+n-1}{n-1} \ge k$.
  This explains that 
  \[ (d+1)d(d-1) \mid c_3(\Pol^d(\C^2)) \quad  \text{ and } \quad (d+1)d \mid c_3(\Pol^d(\C^3)) \]
  in the examples of Remark \ref{rmrk_c3ofpoldc2}.

  As for $d=0$ the $\{x_1,\dots,x_n\}$-degree of $c(\Pol^0(\C^n))$ is $0$, for $k=1$ we
  furthermore get that
  $c_1(\Pol^d(\C^n)) \in \Q[d][x_1,\dots,x_n]^{S_n}$ is divisible by
  $(d+1) d$.
\end{remark}

The comparison of Theorem \ref{thrm_leadingtermcPoldCn} and Corollary
\ref{cor_cPoldCn_Schurcoeffs} shows that
the elementary symmetric polynomial basis differ from the  Schur polynomial
and the monomial symmetric polynomial (see Proposition \ref{prop_cPoldCn_monomsymm}) basis
in that the coefficients of $c_k(\Pol^d(\C^n))$ in the $\{e_\nu: \nu \vdash k \}$ basis
have a non-uniform degree distribution.

For example, the coefficients of $e_1^4$, $e_1^2 e_2$ and $e_2^2$ in
\begin{multline*}
c_4(\Pol^d(\C^2))=(d+1)d(d-1)(d-2)
\Bigg(\frac{(d-3)(15d^3+15d^2-10d-8)}{5760} e_1^4 +
\\
\frac{(d+2)(15d^2-5d-12)}{720}  e_1^2 e_2 + 
\frac{(d+2)(5d+12)}{360}e_2^2 \Bigg)
\end{multline*}
have degrees $8$, $7$ and $6$ respectively.

In Section \ref{subsec_23_345} we provide an upper bound for the degrees of such coefficients.

\bigskip

\subsection{\texorpdfstring{$c\big(\!\Pol^d(\C^2)\big)$}{c(Pold(C2))} as a rising product} \label{sec:cpoldc2-rising}
  In the $n=2$ case $c(\Pol^d(\C^2))=\prod_{t=0}^{d}(1+tx_1+(d-t)x_2)$ is a proper rising
  product, therefore Proposition \ref{prop:rising} provides a reasonable expression for its
  coefficients.
  To illustrate this let us calculate some coefficients $g_{\nu}(d) \in \Q[d]$ in the elementary
  symmetric polynomial basis,
  \[ c(\Pol^d(\C^2))=\sum_{\nu} g_{\nu}(d) e_{\nu}(x_1,x_2). \]

  We will restrict our attention to odd $d$'s, and investigate 
  $c\!\left( \Pol^{2\delta+1}(\C^2)\right)$ for $\delta \in \N$.
  This has the advantage that every term $1+tx_1+(2\delta+1-t)x_2$ of
  $c\!\left( \Pol^{2\delta+1}(\C^2)\right)$ has its ``opposite'' term $1+(2\delta+1-t)x_1+tx_2$,
  and their product---symmetric in $x_1,x_2$---can be written in elementary
  symmetric polynomials $e_i(x_1,x_2)$ as
  \begin{multline*}
    \left( 1+tx_1+(2\delta+1-t)x_2 \right)\left( 1+(2\delta+1-t)x_1+tx_2 \right)=\\
    \underbrace{1+(1+2\delta)e_1+((1+2\delta)t-t^2)e_1^2+
    ((1+4\delta+4\delta^2)+(-4-8\delta)t+4t^2)e_2}_
    {P(\delta,t,e_1,e_2) \in \Q[\delta,t][e_1,e_2]} .
  \end{multline*}
  Then 
  \[ c\!\left( \Pol^{2\delta+1}(\C^2)\right)=\prod_{t=0}^\delta P(\delta,t,e_1,e_2)=
  \sum_H S[P]_H(\delta) e^H=S[P](\delta,e) \]
  is a rising product in the elementary symmetric polynomials $e_1,e_2$, so
  we can use Proposition \ref{prop:rising} to express its Stirling coefficients 
  $S[P]_H(\delta) \in \Q[\delta]$.

  Since 
  $g_{\left(1^{H_1},2^{H_2}\right)}(d)$ and $S[P]_H(\delta)$ are both polynomials,
  \begin{equation}\label{eq_coldc2coeffsvsoddcoeffs}
    g_{\left(1^{H_1},2^{H_2}\right)}(d)=S[P]_H\left(\frac{d-1}{2}\right) \in \Q[d], 
  \end{equation}
  and we get a formula for the coefficients  $g_{\left(1^{H_1},2^{H_2}\right)}(d)$.

  For any exponent vector $H=(H_1,H_2)$ the set of its vector partitions is
  \[ \left\{ J \vdash \left( H_1,H_2 \right) \right\}=
    \left\{ \! \left( (2,0)^j,(1,0)^{H_1-2j},(0,1)^{H_2} \right) : j=0,\dots,
  \left\lfloor \frac{H_1}{2} \right\rfloor \right\}, \]
  therefore \eqref{eq_Stirlingcoeffs_inMtildas} of Proposition \ref{prop:rising} can be written
  as
  \begin{theorem}\label{thm:cPoldC2_elemsymmcoeff} For $c\!\left( \Pol^{2\delta+1}(\C^2)\right)=
  	\sum_H S[P]_H(\delta) e^H$ the coefficients are
  
  \begin{equation*}
    S[P]_H(\delta)=
    \sum_{j=0}^{\lfloor H_1/2 \rfloor} \frac{1}{j! (H_1-2j)! H_2!} 
    \sum_{\lambda \in  \Lambda_j }
    \prod_{s=1}^{j} P_{(2,0),\lambda_s}(\delta)  \!\!
    \prod_{s=j+1}^{H_1-j}  P_{(1,0),\lambda_s}(\delta)  \!\!\!
    \prod_{s=H_1-j+1}^{H_1+H_2-j} \!\!\!\! P_{(0,1),\lambda_s}(\delta) 
    \tilde{M}_\lambda(\delta),
  \end{equation*}
  where $\Lambda_j$ contains the vectors $\lambda=(\lambda_1,\dots,\lambda_{H_1+H_2-j})$ with the propery that
  
  \begin{equation*}
  \lambda_s=
  \begin{cases}
  1,2 & \text{ if } s=1,\dots,j, \\
  0 & \text{ if } s=j+1,\dots,H_1-j, \\
  0,1 \text{ or } 2 & \text{ if } s=H_1-j+1,\dots,H_1-j+H_2,
  \end{cases}
  \end{equation*}
  
  since the coefficients $P_{E,m}(\delta)$ are all zero except for 
  \begin{align*}
    & P_{(1,0),0}(\delta)=2\delta+1,
    \\
    & P_{(2,0),1}(\delta)=2\delta+1, \quad P(\delta)_{(2,0),2}=-1,
    \\
    & P_{(0,1),0}(\delta)=4\delta^2+4\delta+1, \quad P(\delta)_{(0,1),1}=-8\delta-4, \quad
    P(\delta)_{(0,1),2}=4.
  \end{align*}
 \end{theorem}
  \medskip
  For example, for $H=(0,H_2)$, i.e., for the coefficient of $e_2^{H_2}$, Theorem 
  \ref{thm:cPoldC2_elemsymmcoeff} simplifies to 
  \begin{multline*}
  S[P]_{(0,H_2)}(\delta)=
  \frac{1}{H_2!} \sum_{\lambda \in \left\{ 0,1,2 \right\}^{H_2}}
  \prod_{s=1}^{H_2} \left( P_{(0,1),\lambda_s}(\delta) \right)
  \tilde{M}_{\lambda}(\delta)=
  \\
  \frac{1}{H_2!} \sum_{k=0}^{H_2} \sum_{l=0}^{H_2-k} \binom{H_2}{k,l,H_2-k-l}
  4^k (-8\delta-4)^l (4\delta^2+4\delta+1)^{H_2-k-l}
  \tilde{M}_{2^k,1^l,0^{H_2-k-l}}(\delta).
  \end{multline*}
  \medskip
  
  \begin{remark} \label{rmk:duan} 
    Monomial symmetric polynomials can be expressed as polynomials of elementary symmetric polynomials, therefore their arithmetic specialization can be expressed as  polynomials of Stirling numbers. There is no general formula for the coefficients of the transition matrix from  elementary symmetric to  monomial symmetric, making the calculation of general Stirling coefficients difficult. 

    However Duan \cite[Cor.4]{Duan-kostka} gave the following formula in terms of Schur polynomials:
    \begin{equation*}\label{duan-kostka}
      m_{2^l,1^k}=\sum_{t=0}^l(-1)^t\binom{k+t}{t}s_{2^{l-t},1^{k+2t}}.
    \end{equation*}
    This implies that 
    \begin{equation*}\label{mu-in-stirling}
      m_{2^l,1^k}=\sum_{t=0}^l(-1)^t\binom{k+t}{t}\big(\delta_{k+l+t}\delta_{l-t}-\delta_{k+l+t+1}\delta_{l-t-1}\big),
    \end{equation*}
    where we used the abbreviation
    $\delta_k= \delta_k(d)=\stir{d+1}{d+1-k}$.

    This way we obtain a ``closed'' formula for the coefficients of $c\!\left( \Pol^{2\delta+1}(\C^2)\right)$ in the elementary symmetric basis. This is a closed formula if we allow the use of Stirling numbers. The practical use of this formula is limited, it is easier to calculate a given coefficient by computer using interpolation.
  \end{remark}

\medskip

In Section \ref{sec_onclosedformulas} we discuss the feasibility of giving closed formulas
in general and
our motivation for the asymptotic analysis.

\subsection{Proofs of polynomiality  and leading terms of 
  \texorpdfstring{$c_k\big(\Pol^d(\mathbb{C}^n)\big)$}{ck(Pold(Cn))}}\label{subsec_proofscPoldCn}
Theorem \ref{thrm_cpoldCnispoly} and \ref{thrm_leadingtermcPoldCn} expressed in the monomial
symmetric polynomial basis fit well with the theory of rising products, therefore we start by 
proving this version, Proposition \ref{prop_cPoldCn_monomsymm}.
Using the change of basis transformations, it will be then easy to deduce the corresponding
versions in the elementary symmetric polynomial and the Schur polynomial basis.
\begin{proposition}\label{prop_cPoldCn_monomsymm}
  For every partition $\mu \vdash k$, the coefficients $a_\mu(d)$ in
  \[ c_k( \Pol^d(\mathbb{C}^n) )=\sum_{\mu \vdash k} a_\mu(d) m_\mu(x_1,\dots,x_n) \]
  are polynomials  whose leading terms are
  \[ \frac{1}{\mu!} \left( \frac{1}{n!} \right)^{|\mu|} d^{n|\mu|}. \]
\end{proposition}

\begin{proof}
  Notice first that since $c_k(\Pol^d(\mathbb{C}^n))$ is symmetric in the $x_i$ variables, the coefficients in
  \[ c_k(\Pol^d(\mathbb{C}^n)) =\sum_{H} a_H(d) x^H \]
  and in
  \[ c_k( \Pol^d(\mathbb{C}^n) )=\sum_{\mu \vdash k} a_\mu(d) m_\mu(x_1,\dots,x_n) \]
  coincide if $H=\mu$, so we can study the coefficients $a_H(d)$.

  \medskip

  We can write $c(\Pol^{d}(\mathbb{C}^n))$ as an $(n-1)$-long sequence of rising products:
  \begin{equation}\label{eq_cpoldcn_asseriesofproducts}
    \prod_{d_1=0}^{d} \dots \! \prod_{d_i=0}^{d-d_1-\dots-d_{i-1}} \! \ldots \!
    \prod_{d_{n-1}=0}^{d-d_0-\dots-d_{n-2}} \!
    \left( 1+d_1 x_1 +  \dots + d_{n-1} x_{n-1} + ( d-d_1-\dots-d_{n-1} )x_{n} \right),
  \end{equation}
  and we can use results from
  Section \ref{sec_Stirlingcoefs}
  to gather information about its Stirling coefficients.
  Step by step, we arrive at the coefficients of monomial symmetric polynomials.

  In more detail, let
  \begin{equation*}
    \begin{split}
      P^{(0)}(d,d_1,\dots,d_{n-1},x)&=\sum_{H} P^{(0)}_{H}(d,d_1,\dots,d_{n-1}) x^H     \\
      &:=1+d_1 x_1 +  \dots + d_{n-1} x_{n-1} + \left( d-d_1-\dots-d_{n-1}\right)x_{n}.
    \end{split}
  \end{equation*}
  Then for each $i=1,\dots,n-1$ let
  \begin{equation*}
    \begin{split}
      P^{(i)}(d,d_1,\dots,d_{n-i-1},x) & =\sum_{H} P^{(i)}_{H}(d,d_1,\dots,d_{n-i-1})x^H
      \\
      &:=\prod_{d_{n-i}=0}^{d-d_1-\dots-d_{n-i-1}} P^{(i-1)}(d,d_1,\dots,d_{n-i},x)
    \end{split}
  \end{equation*}
  be the rising product
  $S[P^{(i-1)},d-d_1-\dots-d_{n-i-1}](d,d_1,\dots,d_{n-i-1},x)$.
  As for every $d,d_1,\dots,d_{n-i-1}$ in \eqref{eq_cpoldcn_asseriesofproducts}
  $d-d_1-\dots-d_{n-i-1} \ge -1$, 
  we can apply Theorem \ref{thrm_polynomialityofcoeffs_generalsetup} to get 
  that the Stirling coefficients $P^{(i)}_{H}(d,d_1,\dots,d_{n-i-1})$ fit into polynomials
  in $\mathbb{Q}[d,d_1,\dots, d_{n-i-1}]$.
  In particular,
  \[ a_H(d)=P^{(n-1)}_H(d) \in \Q[d] \quad \text{ for } d \ge -1.\]

  \medskip

  Since $\deg(P^{(0)}_E)=|E|$, successive application of Proposition \ref{prop_stepLW} also shows that
  \[ \deg(P^{(i)}_H) \leq i|H|+|H|=(i+1)|H| . \]
  For $i=n-1$ we get that
  \[ \deg(a_H(d)) \le n |H|. \]
  To prove that this degree estimate is sharp, we calculate, by a series of reductions,
  that for every exponent $H=(h_1,\dots,h_n)$
  \[
    a_H(d)=\frac{1}{H!}\left( \frac{1}{n!} \right)^{|H|}d^{n|H|} +\text{ (lower degree terms)}.
  \]

  Throughout our calculations we will use the basic combinatorial identity
  \begin{equation}\label{eq_sumofsumof1sequalbinomial}
    \sum_{d_1=0}^d \sum_{d_2=0}^{d-d_1}\ldots\sum_{d_n=0}^{d-d_1-\dots-d_{n-1}} 1 =
    \binom{d+n}{n},
  \end{equation}
  as well as the fact that for every $n \geq 0, k \geq 1$
  \begin{equation}\label{eq_1overfactorialexpansion}
    \frac{1}{n!}\sum_{t=0}^{n} \binom{n}{t}(-1)^t \frac{1}{t+k} = \frac{(k-1)!}{(n+k)!}.
  \end{equation}
  The latter equation can be proved using induction on $n$.

  According to part \ref{item_propstep2} of Proposition \ref{prop_stepLW} the degree $n|H|$ part of
  $a_H(d)=P^{(n-1)}_H(d)$
  comes from the
  \begin{equation}\label{eq_calcleadingcoeff1}
    \prod_{d_1=0}^d \left( 1 + \sum_{i=1}^n P^{(n-2)}_{1_i}(d,d_1) x_i \right)
  \end{equation}
  summand of $\prod_{d_1=0}^d P^{(n-2)}(d,d_1,x)$. By symmetry reasons and
  \eqref{eq_sumofsumof1sequalbinomial}, this is equal to
  \begin{multline*}\label{eq_calcleadingcoeff2}
    \prod_{d_1=0}^d \Bigg( 1+
      \sum_{d_2=0}^{d-d_1}\ldots \sum_{d_{n-1}=0}^{d-d_1-\dots-d_{n-2}} \!\!\! \left( d_1 \right)x_1+
      \sum_{d_2=0}^{d-d_1}\ldots \sum_{d_{n-1}=0}^{d-d_1-\dots-d_{n-2}} \!\!\!  \left( d_2 \right)
    \left( x_2+\dots+x_n \right) \Bigg) =
    \\
    \prod_{d_1=0}^d \Bigg( 1+\left( \binom{d-d_1+n-2}{n-2} d_1 \right) x_1 +
    \left( \sum_{d_2=0}^{d-d_1} \binom{d-d_1-d_2+n-3}{n-3}d_2 \right) \left( x_2+\dots+x_n \right) \Bigg).
  \end{multline*}
  The highest $\left\{ d,d_1 \right\}$-degree part of the $\sum_{d_2}$ term can be calculated using
  Faulhaber's leading term formula (see also \eqref{eq_leadingtermMtuple}) and  identity
  \eqref{eq_1overfactorialexpansion}. We get that \eqref{eq_calcleadingcoeff1} is further equal to
  \begin{equation*}\label{eq_calcleadingcoeff3}
    \begin{split}
      \prod_{d_1=0}^{d} \Bigg( 1+& \left( \frac{(d-d_1)^{n-2}d_1}{(n-2)!} + \text{ (lower $\{d,d_1\}$-degree terms)}
	\right) x_1 +\\
	&  \left( \frac{(d-d_1)^{n-1}}{(n-1)!}  +\text{ (lower $\{d,d_1\}$-degree terms)} \right)
      \left( x_2+\dots	+x_n\right) \Bigg).
    \end{split}
  \end{equation*}

  Since we are only interested in the leading terms of the $a_H(d)$'s we can drop the lower degree
  terms and continue with the rising product
  \[
    S[P,d](d,x)=
    \prod_{d_1=0}^d
    \Bigg( \underbrace{ 1+\frac{(d-d_1)^{n-2}d_1}{(n-2)!}x_1+\frac{(d-d_1)^{n-1}}{(n-1)!}(x_2+\dots+x_n) }
    _{P(d,d_1,x)}\Bigg)=\sum_H S[P,d]_H(d)x^H.
  \]
  We will use \eqref{eq_Stirlingcoeffs_inMtildas} to compute the leading terms of the Stirling
  coefficients $S[P,d]_H(d)$:
  The coefficients in
  \[ P(d,d_1,x)=\sum_{i=1}^n P_{1_i}(d,d_1)x_i=
  \sum_{i=1}^n \left( \sum_{m \in I_P(1_i)} P_{1_i,m}(d) d_1^m \right) x_i.\]
  are
  \begin{equation*}
    P_{1_i,m}(d)=
    \begin{cases}\displaystyle
      \frac{1}{(n-2)!}\binom{n-2}{m-1}(-1)^{m-1}d^{n-1-m} \quad \left( m=1,\dots,n-1 \right) & \text{ if } i =1, \\[24pt] \displaystyle
      \frac{1}{(n-1)!}\binom{n-1}{m}(-1)^{m}d^{n-1-m} \quad \left( m=0,\dots,n-1 \right)& \text{ otherwise. }
    \end{cases}
  \end{equation*}
  
  For every $H=(H_1,\dots,H_n)$ there is a single partition
  $J=(J_1,J_2,\dots)=\left( 1_1^{H_1},\dots,1_n^{H_n}  \right)$
  with its corresponding term nonzero in \eqref{eq_Stirlingcoeffs_inMtildas}.
  Substituting the leading terms \eqref{eq_leadingtermMtuple} of the
  $\tilde{M}_\lambda(d)$'s into \eqref{eq_Stirlingcoeffs_inMtildas}, we get that
  \begin{multline*}
    S[P,d]_H(d)=\frac{1}{H!} \sum_{\lambda \in \times_{s=1}^{l(J)} I_P(J_s)}
    \left( \prod_{s=1}^{l(J)}\left( \frac{d^{\lambda_s+1}}{\lambda_s+1} \right)
    +\text{ (lower degree terms)} \right) \prod_{s=1}^{l(J)} P_{J_s,\lambda_s}(d)=
    \\
    \frac{1}{H!} \sum_{\lambda \in \times_{s=1}^{l(J)} I_P(J_s)} \left(
      \left( \prod_{s=1}^{l(J)} \frac{d^{\lambda_s+1}}{\lambda_s+1} P_{J_s,\lambda_s}(d) \right)+
      \text{ (lower degree terms)}
    \right)=
    \\
    \frac{1}{H!}
    \left( \frac{d^n}{(n-2)!} \sum_{m=1}^{n-1} \binom{n-2}{m-1} \frac{(-1)^{m-1}}{m+1} \right)^{H_1}
    \left( \frac{d^n}{(n-1)!} \sum_{m=0}^{n-1} \binom{n-1}{m} \frac{(-1)^{m}}{m+1} \right)^{H_2+\dots+H_n}
    +
    \\
    \text{ (lower degree terms)} \stackrel{\eqref{eq_1overfactorialexpansion}}{=}
    \\
    \frac{1}{H!} \left( \frac{1}{n!} \right)^{H_1+\dots+H_n} d^{n|H|}
    +\text{ (lower degree terms)}.
  \end{multline*}
   \end{proof}

   Having completed the hard work of proving the monomial symmetric polynomial version, we can
   easily deduce the results for the elementary symmetry polynomial and then the Schur polynomial
   basis:
  \begin{proof}[Proof of Theorem \ref{thrm_leadingtermcPoldCn}]
  The equivalence with Proposition \ref{prop_cPoldCn_monomsymm} immediately follows from  the multinomial theorem:
    \[ e_{(1^k)}=(x_1+\dots+x_n)^k=\sum_{\mu \vdash k} \binom{k}{\mu_1,\mu_2,\dots} m_\mu. \]
  \end{proof}

  \begin{proof}[Proof of Corollary \ref{cor_cPoldCn_Schurcoeffs}]
  To see the equivalence we need to use the fact (see e.g. \cite[Cor.~7.12.5]{stanley2023enumerative_vol2}) that
    \[ e_1^{k}(x_1,\dots,x_n)=\sum_{\lambda \vdash k} f^{\lambda} s_\lambda(x_1,\dots,x_n), \]
    where $f^\lambda$ denotes the (nonzero) number of standard Young tableaux of shape $\lambda$,
    and that the number of standard Young tableaux of shape
    $\lambda=(\lambda_1,\dots,\lambda_l) \vdash k$
    is known to be (\cite{frobenius1896sitzungsberichte})
    \[ f^\lambda=
    \frac{k! \prod\limits_{1 \leq i < j \leq l} \left( \lambda_i - \lambda_j +j-i\right)}{(\lambda_1+l-1)!(\lambda_2+l-2)!\dots \lambda_l!}. \]
      \end{proof}

\subsection{A degree upper bounds for the coefficients in the elementary symmetric polynomial basis}
\label{subsec_23_345}
This section is dedicated to Theorem \ref{thrm_23_345thrm} that gives a noteworthy upper
bound on the degrees of the coefficient in the elementary symmetric polynomial basis,
improving the bound provided by Theorem \ref{thrm_leadingtermcPoldCn} quite significantly.
In particular, Theorem \ref{thrm_23_345thrm} is no longer equivalent to the theorems and
propositions of the previous section.
Also, we will not use it in the rest of the paper, hence our decision to separate it from the
rest of Section \ref{sec_cpold}. This result seems deeper than the bounds in the other bases. Unfortunately we haven't found enumerative applications. The special case of $n=2$ was noticed by Bal\'azs K\H om\H uves.

\begin{theorem}\label{thrm_23_345thrm}
  Let
  \[ c(\Pol^{d}(\mathbb{C}^n)) = \sum_{\nu} g_\nu(d) e_\nu(x_1,\dots,x_n). \]
  Then for every $\nu=\left( 1^{H_1},\dots,n^{H_n} \right)$
  \begin{equation*}
    \deg\left( g_\nu(d) \right) \le nH_1+(n+1)H_2+\dots+(2n-1)H_n.
  \end{equation*}
\end{theorem}

Let us remark that Theorem \ref{thrm_23_345thrm} provides a significantly better upper bound
than Proposition \ref{thrm_leadingtermcPoldCn}. For example, for the coefficient of
$e_n(x_1,\dots,x_n)$ in
$c\big(\!\Pol^d(\C^n)\big)$ it gives that
\[ \deg(g_{(n)}(d)) \le 2n-1 \]
while Theorem \ref{thrm_leadingtermcPoldCn} asserts that $\deg(g_{(n)}(d)) < n^2$.
Calculations suggest that these upper bounds are sharp:
\begin{conjecture}
  For every $\nu=(1^{H_1},\dots,n^{H_n})$ the leading term of $g_{\nu}(d)$ is
  \[ \prod_{i=1}^n  \frac{1}{H_i!}\left(\frac{(i-1)!}{(n+i-1)!} \right)^{H_i}
  d^{nH_1+(n+1)H_2+\dots+(2n-1)H_n}.  \]
\end{conjecture}
A possible approach to prove this conjecture is to combine the proof of  Theorem \ref{thrm_23_345thrm} and of Proposition \ref{prop_cpoldc2_elemsymmleadingterm}.
\bigskip

Our proof of Theorem \ref{thrm_23_345thrm} is a lengthy, technical application of our theory
of rising products. We postpone it to Appendix \ref{sec_23_345thrm}. Here, we present the proof
(and the calculation of the leading terms) only for the first, $n=2$ case. 
This proof already contains the core ideas for the general statement and also illustrates the
use of Proposition \ref{prop_stepLW}.

\begin{proposition}\label{prop_cpoldc2_elemsymmleadingterm}
  Let 
  \[ c\!\left(\Pol^d(\C^2)\right)=\sum_{\nu} g_\nu(d) e_\nu(x_1,x_2) .\]
  Then for every multiciplity vector $H=(H_1,H_2)$ the coefficient of
  $g_{\left(1^{H_1},2^{H_2}\right)}(d)\in \Q[d]$
  has leading term
  \[ \frac{1}{H_1! H_2!} \frac{1}{2^{H_1+H_2}3^{H_2}} d^{2H_1+3H_2} .\]
\end{proposition}
\begin{proof}
  The proof is based on the approach described in Remark \ref{rmrk_c3ofpoldc2}:
  We restrict our attention to odd $d$'s, take the product of ``opposite'' terms in 
  the elementary symmetric polynomial basis
  \begin{multline*}
    P(\delta,t,e_1,e_2) := 1+(1+2\delta)e_1+((1+2\delta)t-t^2)e_1^2+
    ((1+4\delta+4\delta^2)+(-4-8\delta)t+4t^2)e_2 
  \end{multline*}
  so that we can write
  \[ c\!\left( \Pol^{2\delta+1}(\C^2)\right)=\prod_{t=0}^\delta P(\delta,t,e_1,e_2)=
  \sum_H S[P]_H(\delta) e^H=S[P](\delta,e) \]
  as a rising product.

  Then, we can use Proposition \ref{prop_stepLW} to study the asymptotic behaviour of its Stirling
  coefficients $S[P]_H(\delta) \in \Q[\delta]$. By \eqref{eq_coldc2coeffsvsoddcoeffs}, this leads
  to the leading term of $g_{(1^{H_1},2^{H_2})}(d) \in \Q[d]$.

  Let us adapt the notations of Proposition \ref{prop_stepLW}: The linear form
  \[ W((E_1,E_2)):=E_1+2E_2 \] 
  provides a (sharp) upper bound, $\deg(P_E(\delta,t)) \le W(E)$ for the degrees of the coefficients
  $P_E(\delta,t)$ of $P(\delta,t,e)=\sum_E P_E(\delta,t)e^E$.
  Therefore, by \ref{item_propstep1} of Proposition \ref{prop_stepLW},
  \[ \deg(S[P]_H(\delta)) \leq W(H)+|H|=2H_1+3H_2 \]
  for every $H=(H_1,H_2)$.
  We know, by \ref{item_propstep2} of Proposition \ref{prop_stepLW}, that if
  $\deg(S[P]_H(\delta))$ reaches $2H_1+3H_2$,
  then 
  the leading term of 
  $S[P]_H(\delta)$ comes from the
  \[ \prod_{t=0}^\delta
    \underbrace{1+(1+2\delta)e_1+ ((1+4\delta+4\delta^2)+(-4-8\delta)t+4t^2)e_2}_
    {P'(\delta,t,e_1,e_2)}=S[P'](\delta,e)=\sum_H S[P']_H(\delta) e^H \]
    summand of $S[P](\delta,e)$.

    We will use Proposition \ref{prop:rising}
    to show that for every $H=(H_1,H_2)$ the degree $2H_1+3H_2$ part of $S[P']_H(\delta)$ is
    nonzero, hence, $\deg(S[P]_H(\delta))=W(H)+|H|$ and the leading terms of 
    $S[P]_H(\delta)$ and $S[P']_H(\delta)$ agree.

    First, note that for every $H=(H_1,H_2)$ there is a single partition
    $J=\left( 1_1^{H_1},1_2^{H_2} \right) \vdash H$ (of length
    $l(J)=|H|=H_1+H_2$) such that the corresponding term of
    \eqref{eq_Stirlingcoeffs_inMtildas} is nonzero, hence,
    \begin{equation}\label{eq_cpoldc2elemsymmlt}
      S[P']_H(\delta)=\frac{1}{H_1! H_2!} \sum_{\lambda \in   \N^{|H|} }
      \prod_{s=1}^{|H|} \left( P'_{J_s,\lambda_s}(\delta) \right)\tilde{M}_\lambda(\delta).
    \end{equation}
    As $P'_{E,m}(\delta)=0$ except for 
    \[ P'_{1_1,0}(\delta)=2\delta+1, \quad P'_{1_2,0}(\delta)=4\delta^2+4\delta+1, \quad
    P'_{1_2,1}(\delta)=-8\delta-4, \quad P'_{1_2,2}(\delta)=4, \]
    the $\lambda$'s giving nonzero terms of \eqref{eq_cpoldc2elemsymmlt} all have coordinates
    \begin{equation*}
      \lambda_s=
      \begin{cases}
	0 & \text{ if } s=1,\dots,H_1, \\
	0,1 \text{ or } 2 & \text{ if } s=H_1+1,\dots,H_1+H_2.
      \end{cases}
    \end{equation*}
    Therefore, \eqref{eq_cpoldc2elemsymmlt} is further equal to
    \begin{equation}\label{eq_cpoldc2elemsymmlt2}
      S[P']_H(\delta)=\frac{1}{H_1! H_2!}  \sum_{\sigma \in \left\{ 0,1,2 \right\}^{H_2}}
      P'_{1_1,0}(\delta)^{H_1}
      \prod_{s=1}^{H_2} \left( P'_{1_2,\sigma_s}(\delta) \right)
      \tilde{M}_{0^{H_1},\sigma}(\delta).
    \end{equation}

    To extract the leading term, let us first recall that, by Lemma
    \ref{lemma_tildeMlambdaspolys}, the leading term of $\tilde{M}_{0^{H_1},\sigma}(\delta)$ is
    \[ \left( \frac{\delta^{0+1}}{0+1}\right)^{H_1} \prod_{s=1}^{H_2}
    \frac{\delta^{\sigma_s+1}}{\sigma_s+1}. \]
    As for every $m \in \{0,1,2\}$ $\deg(P'_{1_2,m}(\delta))=2-m$,
    we see that for every
    $\sigma \in \left\{ 0,1,2 \right\}^{H_2}$ the degree of the corresponding term in 
    \eqref{eq_cpoldc2elemsymmlt2} is the same,
    \[ (1+0+1)H_1+\sum_{s=1}^{H_2}
    (2-\sigma_2)+(\sigma_2+1)=2H_1+3H_2. \]
    This means that the leading coefficient of $S[P^{'}]_H(\delta)$ is the sum of the leading
    coefficients of the terms corresponding to all $\sigma$'s.
    By looking at the leading coefficients of the $P'_{1_i,m}(\delta)$'s and the 
    $\tilde{M}_{0^{H_1},\sigma}(\delta)$'s, we see that this sum is
    \[ \frac{1}{H_1! H_2!} \left(2 \cdot \frac{1}{0+1} \right)^{H_1} \left(4 \cdot \frac{1}{0+1} 
    -8 \cdot \frac{1}{1+1}+4 \cdot \frac{1}{2+1} \right)^{H_2}=
  \frac{1}{H_1! H_2!} 2^{H_1} \left( \frac{4}{3} \right)^{H_2} .\]

  Finally, by \eqref{eq_coldc2coeffsvsoddcoeffs}, we get that
  \[ \text{the leading term of } g_{\left(1^{H_1},2^{H_2}\right)}(d)=
  \frac{1}{H_1! H_2!} \left( \frac{1}{2} \right)^{H_1+H_2}\left( \frac{1}{3} \right)^{H_2}. \]
  \end{proof}

  \bigskip

  Let us remark here, that by
  a step-by-step generalization of the previous proof, 
  part \ref{item_propstep2} of Proposition \ref{prop_stepLW} can be easily strengthened:
  \begin{proposition}\label{prop_LWleadingcoeff}
    If there is a linear form $W: \mathbb{Z}^\infty \to \mathbb{Z}$ such that for every 
    $E \in \mathbb{N}^\infty$, $\deg(P_E(d,t)) \leq W(E)$. Then
    $\deg(S[P,L]_H(d))$ reaches $W(H)+|H|$ for every $H$ if and only if
    \begin{multline*}
      \deg(P_{1_i}(d,t))=W(1_i) \quad \text{ and }
      \sum_{\substack{m: \\ \deg(P_{1_i,m}(d))=W(1_i)-m}}
      \left( \frac{1}{1+m}
      \coef\left(d^{W(1_i)-m}, P_{1_i,m}(d)\right) \right) \neq 0 \\
      \text{ for every } i=1,\dots,n.
    \end{multline*}
    In this case, the leading coefficient of $S[P,L]_H(d)$ is
    \[ \frac{1}{H!}\prod_{i=1}^n \left( \sum_{\substack{m: \\ \deg(P_{1_i,m}(d))=W(1_i)-m}}
	\left( \frac{1}{1+m}
      \coef\left(d^{W(1_i)-m}, P_{1_i,m}(d)\right) \right) \right).
    \]
  \end{proposition}

  \subsection{On closed formulas and generating functions}\label{sec_onclosedformulas}
  For any given $k$ and $n$ we showed that $c_k(\Pol^{d}(\mathbb{C}^n))$ is a polynomial and we also have a degree bound so we can interpolate. In fact, this is how we calculated e.g. the formulas of Remark \ref{rmrk_c3ofpoldc2}.
  
  \medskip 

  In Corollary \ref{cor_cPoldCn_Schurcoeffs} and Proposition \ref{prop_cPoldCn_monomsymm} we gave
  formulas for the leading terms of the coefficients of $c(\Pol^{d}(\mathbb{C}^n))$ in
  terms of the corresponding partition.
  One may wonder if we can go beyond asymptotics and give such closed formulas also for
  the coefficients of $c(\Pol^{d}(\mathbb{C}^n))$. 

  In some sense even the Stirling numbers, the simplest Stirling coefficients, cannot be
  expressed by closed formulas. To obtain closed formulas, we need to fix our ``atoms'',
  expressions which we treat as building blocks.
  In Section \ref{sec_Stirlingcoefs} we chose arithmetic specializations of symmetric functions
  as ``atoms'' in our formulas.
  Theorem \ref{thm:cPoldC2_elemsymmcoeff} and Remark \ref{rmk:duan} also
  provide, in this sense, closed formulas for the coefficients of $c(\Pol^{d}(\mathbb{C}^2))$.
  
  \medskip

   An other possibility would be to give the answer in terms of a generating function.
   But in our case we started with the generating function \eqref{non-rising}. This is an example when it is difficult to obtain the desired information from the generating function.
   
   \medskip
    
  In \cite[Thm.~5.1, Cor 5.7]{feher2023plucker} we gave a closed formula for the Euler class
$e(\Pol^{d}(\mathbb{C}^2))=c_{d+1}(\Pol^{d}(\mathbb{C}^2))$ in terms of Stirling numbers:

\begin{theorem}\label{thm-euler}
  The coefficients of Schur polynomials in $c_{d+1}\!\left( \Pol^d(\C^2) \right)$ are
  \begin{multline*}
    e^d_j:=\coef\left( s_{d+1-j,j}, c_{d+1}\left( \Pol^d\left( \C^2 \right) \right)\right)=\\[2mm]
    \begin{cases}
      0 & \text{ if } j=0, \\[8pt]
      \begin{multlined}
	\sum_{k=j-1}^{d-j} \left( (-1)^{k+j-1}\binom{k}{j-1}\stir{d}{d-k}d^{d+1-k} \right) + \\
	\sum_{k=d-j+1}^{d-1} \left(\left( (-1)^{k+j-1}\binom{k}{j-1}-
	(-1)^{d+1+k-j}\binom{k}{d-j+1}\right) \stir{d}{d-k} d^{d+1-k} \right)
      \end{multlined}& \text{ otherwise. } \\
    \end{cases}
  \end{multline*}
\end{theorem}

\subsection{The $n$-dependence of \texorpdfstring{$c_k\big(\Pol^d(\mathbb{C}^n)\big)$}{ck(Pold(Cn))}}
\label{sec:separately}
In \cite{ciliberto_zaidenberg2020onfanoschemes} $c_k\big(\!\Pol^d(\C^n)\big)$ is calculated for $k=1$ and $k=2$:
\begin{equation*}
c_1\big(\!\Pol^d(\C^n)\big)=\binom{d+n-1}{n}e_1
\end{equation*}
and
\begin{equation*}
c_2\big(\!\Pol^d(\C^n)\big)=\binom{d+n}{n+1}e_2+
\left(\frac{1}{2}\binom{d+n-1}{n}^2-\frac{1}{2}\binom{d+n-1}{n}-\binom{d+n-1}{n+1}\right)e_1^2.
\end{equation*}

Based on these and some further computer aided calculations we formulated two conjectures:

\begin{conjecture}[Weak conjecture]
$c_k\big(\!\Pol^d(\C^n)\big)$ is a polynomial in $n$ for any given $k$ and $d$.
\end{conjecture}
Notice that this conjecture implies that $c_k\big(\!\Pol^d(\C^n)\big)$ is \emph{separately polynomial} in $n$ and $d$
for any $k$, but not a polynomial in $n,d$.

\begin{conjecture}[Strong conjecture]
$c_k\big(\!\Pol^d(\C^n)\big)$ is of the form $P(B_1,\dots,B_s)$, where $P$ is a polynomial and $B_i$ are binomial coefficients of the form $\binom{d+n+a_i}{n+b_i}$.
\end{conjecture}

\section{Enumerative applications} \label{sec:enum}
In this Section we apply our results on $c\big(\!\Pol^d(\C^n)\big)$ to calculate invariants of
two types of varieties related to degree $d$ projective hypersurfaces and projective $r$-planes
contained in them.

These two types correspond to opposite situations
that can be explained using the following fairly standard diagram, see
\cite[Ch.~12; Ch.~22]{harris1992alggeom} for more details:
Consider the incidence variety with its projections
into the space of degree $d$ hypersufaces and into the space of projective $r$-spaces of $\P^m$:

\begin{equation}\label{tikz_incidencevar}
  \begin{tikzcd}[column sep=0em]
    & \phi:=\left\{ (Z,\Delta) \in \P(\Pol^d(\C^{m+1})) \times \Gr_{r+1}(\C^{m+1}): \Delta \subset 
    Z \right\} \ar[dl,"\pi_{\P}"] \ar[dr,"\pi_{\Gr}"]&
    \\
    \P(\Pol^d(\C^{m+1})) & & \Gr_{r+1}(\C^{m+1})
  \end{tikzcd}
\end{equation}
Then,
either $\pi_{\P}$ is surjective and we can
look at the fiber $\pi_{\P}^{-1}(Z) \subset \Gr_{r+1}(\C^{m+1})$ for a generic hypersuface $Z$
(the variety of $r$-spaces inside $Z$),
or $\pi_{\P}$ maps to a proper subvariety (the variety of degree $d$ hypersurfaces in $\P^m$
  containing an $r$-space) that we can investigate.

From the second projection $\pi_{\Gr}$ one can easily deduce that
\[ \dim(\phi)=\dim\big(\P(\Pol^d(\C^{m+1}))\big)+\delta(d,m,r),\] where
\[ \delta(d,m,r):= (r+1)(m-r)-\binom{r+d}{r} \]
is called the \emph{expected dimension} of the general fiber of $\pi_{\P}$.
The sign of the expected dimension can help to differentiate the two cases:

If $\delta(d,m,r) < 0$ the projection $\pi_{\P}$ cannot be surjective.
For the $d \ge 3$ the contrary also holds: $\delta(d,m,r) \ge 0$ implies that $\pi_{\P}$ is
surjective (with $\delta(d,m,r)$ the actual dimension of its general fibers).
For $d=2$  a general quadric hypersurface contains an $r$-plane precisely if $r$ is at
most half of its dimension.  

 For example, a general quadric in $\P^4$ does not contain any 2-planes, even
though $\delta(2,4,2)=0$.

\subsection{Degree of varieties of hypersurfaces containing linear subspaces}
\label{subsec_hypersurfswithlinsubspaces}
Let us start with the case where $\pi_{\P}$ of \eqref{tikz_incidencevar} is not surjective:
Denote by
\[ \Sigma(d,m,r) \subset \P(\Pol^{d}(\mathbb{C}^{m+1})) \]
the subvariety whose points correspond to degree $d$ hypersurfaces that do contain a
projective $r$-plane.
The study of these varieties was initiated by Manivel. In \cite{manivel1999surleshypersurfaces}
he proved that if $\delta(d,m,r)<0$ and $d \ge 3$, then 
$\Sigma(d,m,r)$ is an irreducible subvariety of codimension $-\delta(d,m,s)$ 
in $\P(\Pol^{d}(\mathbb{C}^{m+1}))$, its generic point corresponds to a degree $d$ hypersurface
that carries a unique $r$-plane and its degree is
 \begin{equation}\label{eq_degreeSigma}
   \deg(\Sigma(d,m,r))= \int_{\Gr_{r+1}(\mathbb{C}^{m+1})} c_{(r+1)(m-r)} (\Pol^d(S)),
 \end{equation}
 where $S \to \Gr_{r+1}(\mathbb{C}^{m+1})$ denotes the tautological bundle.
The $d=2$ case is again an exception: the easiest example is those of quadric plane curves, where
$\delta(2,2,1)=-1$, but a general element of $\Sigma(2,2,1)$ contains exactly two lines.

 Though the proof of degree formula can be found in
 \cite{manivel1999surleshypersurfaces} (or in \cite{ciliberto_zaidenberg2020onfanoschemes}),
 we decided to give an other one
 that uses a general simple fact which doesn't seem to be widely known.

 \begin{proposition}\label{prop:deg-of-embedded-reso}
   Suppose that $M$ is a smooth projective variety of dimension $b$ and $E\to M$ is a subbundle
   of the trivial vector bundle $M\times V$. Suppose that the projection $\pi_V: M\times V\to V$
   restricts to a resolution $E\to Y:=\pi_V(E)$. Then the degree of $\P(Y)\subset \P(V)$ is given
   by
 \begin{equation*}\label{deg-of-embedded-reso}
  \int\limits_{M}\!\! c_b(V/E)=\deg Y=\deg\P Y,
\end{equation*} 
where we abbreviate by $V$ the trivial vector bundle with fiber $V$.
\end{proposition}
\begin{proof}
  Let us consider the projectivization of our vector bundles,
  \begin{equation*}\label{eq_incidencevar_dagram}
    \begin{tikzcd}
      \P(E) \ar[r,hookrightarrow] \ar[dr] & M \times \P(V) \ar[r,"\P(\pi_V)"] \ar[d] & \P(V)
      \\
      & M &
    \end{tikzcd}
  \end{equation*}
  and denote by $\gamma\to \P(V)$ the tautological line bundle. The bundle
  $\Hom(\gamma,V/ E) \to M \times \P(V)$ has a tautological section
  defined by $s(m,[v])(x):=x+E_m$ for $x\in \gamma_{[v]}$.
  This section is transversal by a standard argument (see \cite[A.2]{feher2023plucker} for a similar situation)
  to the zero section and $s^{-1}(0)=\P(E)$. This implies that
  \[ [\P(E)\subset M \times \P(V)]=e\big(\Hom(\gamma,V/E)\big)=\sum c_i(V/E)u^{k-i},\]
  where $u=c_1(\gamma^\vee)\in \P V$ is the hyperplane class and $k=\rank(V/E)$.

  Moreover, the restriction of $\P(\pi_V): M \times \P(V) \to \P(V)$ to $\P(E)$ is a resolution
  of $\P(Y)$, therefore
  \[ [\P(Y) \subset \P(V)]=\P(\pi_V)_![\P(E)\subset M \times \P(V)]=
  \P(\pi_V)_!\left(\sum c_i(V/E)u^{k-i}\right)=u^{k-b}\int\limits_{M}\!\! c_b(V/E).\]
\end{proof}

To deduce the degree formula \eqref{eq_degreeSigma} from Proposition
\ref{prop:deg-of-embedded-reso} let us consider the evaluation map
\[ ev:\gr_{r+1}(\C^{m+1}) \times \Pol^d(\C^{m+1}) \to \Pol^d(S^{r+1}), \quad
(\Delta,f) \mapsto f|_\Delta. \]
The projectivization of the projection of its kernel bundle
$E:=\ker(ev) \subset  \gr_{r+1}(\C^{m+1}) \times \Pol^d(\C^{m+1}) $ to $\Pol^d(\C^{m+1})$ is
exactly the projection $\pi_{\P}: \phi \to \P(\Pol^d(\C^{m+1}))$ of the incidence variety in
\eqref{tikz_incidencevar}.
From \cite{manivel1999surleshypersurfaces} we obtain that if $\delta(d,m,r) < 0$ and $d \ge 3$,
the projection $\pi_{\Pol^d(\C^{m+1})}$ restricted to $E$ is a resolution of the cone of
$\Sigma(d,m,r)$. Therefore, we can apply Proposition \ref{prop:deg-of-embedded-reso} to get
\eqref{eq_degreeSigma}.

\medskip

Now notice that the characteristic class $c(\Pol^d(\C^{r+1}))$ evaluates to $c(\Pol^d(S^{r+1}))\in H^*(\Gr_{r+1}(\mathbb{C}^{m+1}))$.
As integration over $\Gr_{r+1}(\mathbb{C}^{m+1})$ returns the coefficient of the volume form
$s_{(m-r)^{r+1}}$,
we get that
\begin{equation*}\label{eq_degSigma_ascoef}
\deg(\Sigma(d,m,r))=\coef(s_{(m-r)^{r+1}}, c(\Pol^d(\mathbb{C}^{r+1})).
\end{equation*}

This, combined with Theorem \ref{cor_cPoldCn_Schurcoeffs}, implies that
\begin{theorem}\label{thm:degSigma_leadingterm}
  Let $\Sigma(d,m,r)$ denote the subvariety of $\P(\Pol^{d}(\mathbb{C}^{m+1}))$ whose points
  correspond to  degree $d$ hypersurfaces in $\P^m$ that contain an $r$-plane.
  Let us fix $m$ and $r$. Then, for big enough $d$'s (i.e. for $d \ge 3$ and $\delta(d,m,r)<0$)
  the degrees $ \deg(\Sigma(d,m,r))$ form a polynomial in $d$, whose leading term is
  \[ \frac{1!\cdot2!\cdots(r-1)!r! }{m!\dots(m-r)!}
  \left( \frac{1}{(r+1)!} \right)^{(r+1)(m-r)} d^{(r+1)^2(m-r)}. \]
\end{theorem}

\begin{example}
  For $m=3$, $r=1$ and $d \ge 4$ Theorem \ref{thm:degSigma_leadingterm} tells us that
  $\deg(\Sigma(d,3,1))$,
  the degree of the variety of degree $d$ surfaces in $\P^3$ containing a projective line,
  is a degree 8 polynomial in $d$. By interpolating for e.g. $d=0,\dots,8$ the values of
  $\coef(s_{2,2}, c(\Pol^d(\mathbb{C}^{2}))$---see Remark \ref{rmrk_polynomialityfrom0}---,
  we get that
  \begin{equation*}
    \begin{split}
    \deg(\Sigma(d,3,1))&=
    \frac{1}{192}d^{8}+\frac{1}{288}{d}^{6}-\frac{1}{48}d^{5}-\frac{25}{576}d^{4}-
    \frac{1}{16}{d}^{3}+\frac{5}{144}d^2+\frac{1}{12}d\\
    &=\frac{1}{576}d(d-1)(d-2)(d+1)\left( 3d^3+6d^3+17d^2+22d+24 \right).
  \end{split}
  \end{equation*}
\end{example}

\begin{remark} The case of varieties $\Sigma(d,m,m-1)$ is somewhat degenerate: a hypersurface is in $\Sigma(d,m,m-1)$ if and only if it has a linear component. This description provides us with a simple resolution:
 
  The map $\varphi: \P(\Pol^{1}(\mathbb{C}^{m+1})) \times \P(\Pol^{d-1}(\mathbb{C}^{m+1})) \hookrightarrow
  \P(\Pol^{d}(\mathbb{C}^{m+1}))$ induced by multiplication of polynomials provides a resolution of
  $\Sigma(d,m,m-1)$. Such a resolution can be used to calculate the degree as
  \[ \deg(\Sigma(d,m,m-1))=\int_{\P(\Pol^{1}(\mathbb{C}^{m+1})) \times \P(\Pol^{d-1}(\mathbb{C}^{m+1}))}
  \varphi^* c_1^{D}, \]
  where $D=\binom{d+m-1}{m}+m-1$ is the dimension of the domain and $c_1=c_1(\gamma^\vee)$ for $\gamma$
  the tautological line bundle.
  Let $u$ and $v$ denote the first Chern classes of the duals of tautological bundles over the first and
  second factor of the domain.
  As $\varphi^* \gamma^\vee$ is the tensor product of these dual bundles,
  \[ \varphi^* \left( c_1^{D} \right)=(u+v)^{D}=
  \sum_{t=0}^{D} \binom{D}{t} u^t v^{d-t}. \]
  This integration amounts to taking the coefficient of the volume form $u^m v^{D-m}$, hence, substituting
  $D=\binom{d+m-1}{m} +m-1$, we get that
  \[ \deg(\Sigma(d,m,m-1))=\binom{\binom{d+m-1}{m}+m-1}{m}. \]
\end{remark}

\begin{remark} There are several ways to calculate an integral over the Grassmannian, see e.g.
  \cite{ciliberto_zaidenberg2020onfanoschemes}.
  Even our very first description \eqref{eq_cPoldCn_singleproduct} of $c(\Pol^d(\mathbb{C}^{r+1}))$
  can be used to obtain formulas for $\deg(\Sigma(d,m,r))$ that can be calculated via a computer 
  program. However, these formulas don't tell us much about the $d$-dependence of the degree $\deg(\Sigma(d,m,r))$.
  For example,
  in \cite{manivel1999surleshypersurfaces} Manivel, based on computations, claims that fixing
  $m$ and $r$, $\deg(\Sigma(d,m,r))$
depends on $d$ as a degree $2(r+1)(m-r)$ polynomial. Comparing with Theorem
\ref{thm:degSigma_leadingterm}, we can see that this holds only for $r=1$.
The big advantage of having a priori knowledge on the polynomiality and the degree is that we can use interpolation to calculate these polynomials.
\end{remark}

\subsection{Degree of Fano schemes}
\label{sec:deg-of-fano}
Let us turn to the case where $\pi_{\P}$ of \eqref{tikz_incidencevar} is surjective:
The Fano scheme of $r$-planes of a projective variety $Y \subset \P^m$
is defined to be
the subvariety of the Grassmannian $\Gr_{r+1}(\mathbb{C}^{m+1})$ containing the projective
$r$-planes inside $Y$.
The Fano scheme of a hypersurface $Z \subset \P^m$ is just the fiber of $\pi_{\P}$ of
\eqref{tikz_incidencevar} over $Z$.

For a generic degree $d$ polynomial $f$ the corresponding section $\sigma_f$ of $\Pol^d(S)$ (defined by $\sigma_f(V)=f|_V$) is transversal to the zero section, therefore the cohomology fundamental class of its zero locus, the Fano scheme  $F_r(d,m)$ is
\[ \left[ F_r(d,m) \right]=e\big(\Pol^d(S)\big) \]
and, hence,
\begin{equation*}\label{eq_fano-degree}
  \deg(F_r(d,m))=\int_{\Gr_{r+1}(\mathbb{C}^{m+1})}e(\Pol^d(S))c_1^\delta(S^\vee),
\end{equation*}
where $S \to \Gr_{r+1}(\mathbb{C}^{m+1})$ denotes the tautological bundle
and $\delta=\delta(d,m,r)$ is the dimension of the Fano scheme.
Here the degree is defined as the projective degree of the image of the Fano scheme under the
Pl\"ucker embedding.

This formula is well known, and there are several methods to calculate the integrand and the integral as well. 
The novelty here is that we can use results of Section \ref{sec_onclosedformulas} to obtain closed
formulas for the degrees.

$e(\Pol^d(S))=c_{d+1}(\Pol^d(S))$,  so we are not looking for a Chern class independent of $d$. Therefore we cannot use our results on the
polynomiality of the coefficient.
In fact, the degree of Fano schemes of generic hypersurfaces is an example of an enumerative
problem where the $d$-dependence is not polynomial.

\begin{remark}
	For $d=2$ it is proved in \cite{lalat} that
	\[e(\Pol^2(\C^{r+1}))=2^{r+1}s_{(r+1,r,\dots,2,1)}, \]
which implies that $e(\Pol^2(\C^{r+1}))$ is nonzero exactly if $(r+1,r,\dots,2,1)\subset (m-r)^{r+1}$. The cohomology class of a subvariety of the Grassmannian is zero only if it is empty, therefore we gave a cohomological proof of the statement at the previous section: a general quadric hypersurface contains an $r$-plane precisely if $r$ is at
most half of its dimension. 
\end{remark}

\medskip

The integral over $\Gr_{r+1}(\C^{m+1})$ returns the coefficient of the volume form
$s_{(m-r)^{r+1}}$. In the $r=1$ case Theorem \ref{thm-euler} provides closed formulas for the
coefficients of $e(\Pol^d(\C^2))$ in the Schur polynomial basis
from which we can deduce a
closed formula for the degrees of Fano schemes of lines:

\begin{theorem} \label{thm-deg-of-fano-of-lines} For $\delta=2m-d-3\geq0$ the  degrees of Fano schemes of lines are given by
	
\[ \deg(F_1(d,m))= \sum_{j=0}^{\left\lfloor \frac{\delta}{2}\right\rfloor} C(\delta-j,j)e^{d}_{d-m+2-j},
\]
where	
\[C(\delta-j,j)=
\binom{\delta}{j}-\binom{\delta}{j-1}=\frac{\delta!(\delta-2j+1)}{j! (\delta-j+1)!} \]
are the Catalan triangle numbers and $e^{d}_{d-m+2-j}$ was defined in Theorem \ref{thm-euler}.
\end{theorem}
\begin{example}\label{fano-curve-degree}
For example, the Fano scheme of lines is ($\delta=)$1-dimensional precisely if $d=2m-4$,
and we get that
the degree of the Fano curve is
\begin{multline*}
\deg\left(F_1\left(2m-4,m\right)\right)=\coef\left(s_{\left(m-1,m-2 \right)},
e\left(\Pol^{2m-4}(\C^{2})\right)\right)=
\\[6pt]
\shoveleft{\left( \stir{2m-4}{m-1}-\frac{1}{2}\stir{2m-4}{m-2}\right) (2m-4)^{m}+}
\\
\sum_{k=m-1}^{2m-5}(-1)^{m-1+k} \left( \binom{k}{m-3}
- \binom{k}{m-1}\right)\stir{2m-4}{2m-4-k}(2m-4)^{2m-3-k}.
\end{multline*}
The first values for $m=3,4,5,6,7,\dots$ are $4, 320, 60480, 21518336, 12493096000,\dots$.
\end{example}

\begin{proof}[Proof of Theorem \ref{thm-deg-of-fano-of-lines}]
For each Schur polynomial $s_{m-1-j,m-1-\delta+j}$ in the Schur expansion of $e(\Pol^d(\C^2))$
there is a unique Schur polynomial $s_{\delta-j,j}$ in the Schur expansion of $c_1^\delta$ such that
their product is $s_{(m-1)^{2}}$. This gives that

\begin{equation*}
\deg(F_1(2m-\delta-3,m))=\int_{\Gr_2(\C^{m+1})}e(\Pol^{d}(\C^2))c_1^{\delta}=
\sum_{j=0}^{\left\lfloor \frac{\delta}{2} \right\rfloor}
\coef\left( s_{\delta-j,j},c_1^\delta  \right) e^{d}_{m-1-\delta+j},
\end{equation*}
where
\[e^{d}_{m-1-\delta+j}=\coef\left( s_{m-1-j,m-1-\delta+j}, e(\Pol^{d}(\C^2)) \right)\]
was calculated in Theorem \ref{thm-euler}.
The  coefficients $ \coef\left( s_{\delta-j,j},c_1^\delta  \right)$, Kostka numbers by definition, are known to be Catalan triangle numbers:
\[ \coef\left( s_{\delta-j,j},c_1^\delta  \right)=K_{(\delta-j,j),1^\delta}=C(\delta-j,j)=
\binom{\delta}{j}-\binom{\delta}{j-1}=\frac{\delta!(\delta-2j+1)}{j! (\delta-j+1)!}. \]
In fact, these Kostka numbers were calculated by Schubert as the degrees of the corresponding Schubert varieties. 	
\end{proof}

 \bigskip

\subsection{Euler characteristics of Fano schemes} \label{sec:euler}

For a generic hypersurface the Fano scheme is smooth and its normal bundle is isomorphic to the restriction of $\Pol^d(S)$. Using the Poincar\'e-Hopf theorem we obtain the Euler characteristics

\begin{equation} \label{fano-chi}
  \chi(F_r(d,m))=
  \int_{\Gr_{r+1}(\mathbb{C}^{m+1})}    c(\Gr_{r+1}(\mathbb{C}^{m+1})) \frac{e(\Pol^d(S^{r+1}))}{c(\Pol^d(S^{r+1}))}.
 \end{equation}
 
 This integral can be calculated for any given values of the parameters, but we don't expect a closed formula. But for certain families we can give closed formulas. We want to use our result Theorem \ref{thm-euler}, so we fix $r=1$: study Fano schemes of lines. Similarly to the degree calculations we can fix $\delta$, the dimension of the Fano scheme. For $\delta$ at most $3$, we need the first three Chern classes of $\Pol^d(\C^{2})$---these are given in Remark \ref{rmrk_c3ofpoldc2}---and the first 
three Chern classes of $\Gr_{2}(\mathbb{C}^{m+1})$.

\subsubsection{Total Chern classes of Grassmannians} The tangent space
$T(\gr_k(\C^n))$ of the Grassmannian is isomorphic to $\Hom(S,\C^n/S)$, which implies that topologically $T(\gr_k(\C^n))\oplus\Hom(S,S)=\Hom(S,\C^n)$. Therefore, we have
\begin{equation*}\label{cgr}
c(\gr_k(\C^n))=\frac{\prod\limits_{i=1}^{k}(1+x_i)^n}{\prod\limits_{i=1}^{k}\prod\limits_{j=1}^{k}(1+x_j-x_i)},
  \end{equation*}
\smallskip

where $x_i$ are the Chern roots of $S^\vee$, i.e. $c_j(S^\vee)=\sigma_j(x_1,\dots,x_k)$. 

\begin{example}\label{key} For $k=2$ we have
	
	\[c(\gr_2(\C^n))=\frac{(1+c_1+c_2)^n}{1-c_1^2+4c_2},\]
	implying that	
	
	\[c_1(\gr_2(\C^n))=nc_1,\ \ \ \  c_2(\gr_2(\C^n))=
	\frac{1}{2}(n^2-n+2)s_2+\frac{1}{2}(n^2+n-6)s_{1,1} \]
	and
	\[c_3(\gr_2(\C^n))=\frac{1}{3}n(n^2-7)s_{2,1}+
	\frac{1}{6}n(n^2-3n+8)s_{3}. \]
	\end{example}
	
For $\delta=1$, when the Fano scheme is a curve, we have $d=2m-4$ and \eqref{fano-chi} specializes to
\begin{equation}\label{fano-chi-delta1}
  \chi(F_1(2m-4,m))=e^{2m-4}_{m-2}\left(m+1-\binom{2m-3}{2}\right),
\end{equation}
using Remark \ref{rmrk_c3ofpoldc2} and Example \ref{key}.
Compared with the degree calculated in Example \ref{fano-curve-degree} we can see that \[\chi(F_1(2m-4,m))=\left(m+1-\binom{2m-3}{2}\right)\deg(F_1(2m-4,m)).\]
This connection was noticed in \cite{hiep2017numericalinvariantsfanoschemes}.

For $\delta=2$, when the Fano scheme is a surface, we have $d=2m-5$ and \eqref{fano-chi} 
specializes to

\begin{equation}\label{fano-chi-delta2}
\begin{split}
\chi(F_1(2m-5,m))=\,&
e^{2m-5}_{m-2}(2{m}^{4}-20{m}^{3}+67{m}^{2}-86m+36)    \\
+\,&e^{2m-5}_{m-3}(2{m}^{4}-\frac {56{m}^{3}}{3}+59{m}^{2}-{\frac{208m}{3}}+23).
\end{split}
\end{equation}
The Euler characteristics of Fano surfaces of lines was calculated in \cite[ex. 5.2]{ciliberto_zaidenberg2020onfanoschemes}, the novelty here is that we give a closed formula. Similarly one can obtain closed formulas for $\chi(F_1(2m-3-\delta,m))$ for any given $\delta$.

\begin{remark}The main obstacle to find further closed formulas is the lack of closed formulas for $e(\Pol^d(\C^{r+1}))$ for $r>1$. The other constituents of \eqref{fano-chi} can be calculated via interpolation using the polynomial properties of $c_i(\Pol^d(\C^{r+1}))$ and $c_i(\gr_{r+1}(\C^{m+1}))$.

\end{remark}

\appendix
\section{Proof of Theorem \ref{thrm_23_345thrm}}
\label{sec_23_345thrm}
In this Section we will use our theory of rising products and Stirling coefficients to prove
\medskip

\noindent\textbf{Theorem \ref{thrm_23_345thrm}.}
\emph{
  Let
  \[ c( \Pol^d(\mathbb{C}^n) )=\sum_{\nu} g_\nu(d) e_\nu(x_1,\dots,x_n). \]
  Then for every $\nu=\left( 1^{H_1},\dots,n^{H_n} \right)$
  \begin{equation*}
    \deg\left( g_\nu(d) \right) \le nH_1+(n+1)H_2+\dots+(2n-1)H_n.
  \end{equation*}
}

\medskip

The main difficulty is that although 
\[ c\big(\!\Pol^d(\C^n)\big)= \prod_{\substack{(i_1,\dots,i_n) \\ i_1+\dots+i_n=d}}
\left( 1+i_1 x_1+\dots+i_n x_n \right) \in \Q[d][[x_1,\dots,x_n]]^{S_n}, \]
the individual terms
of this factorization are in general not symmetric, hence cannot be written in the elementary
symmetric polynomial basis. To remedy this problem, we will group these terms
corresponding to orbits of the $S_n$-action.
Much like when we grouped ``opposite'' terms of $c\big(\!\Pol^d(\C^2)\big)$ in the proof of Proposition
\ref{prop_cpoldc2_elemsymmleadingterm}.

The problem is that the $S_n$-orbits of the terms of $c\big(\!\Pol^d(\C^n)\big)$ are much more diverse than
those for $n=2$.
This makes this proof rather long, so we decided to divide it into several parts.

\subsection{A factorization of \texorpdfstring{$c\big(\!\Pol^d(\C^n)\big)$}{c(Pold(Cn))} corresponding to orbit types}
Let us denote the
$d$-fold dilation of the standard $(n-1)$-simplex by
\[ d\Delta^{n-1}:=\left\{ \left. (i_1,\dots,i_n) \in \R^n \ \right| \sum_{t=1}^n i_t=d \text{ and } i_t \ge
  0 \text{ for } t=1,\dots,n \right\}. \]
Its integer lattice points, $d\Delta^{n-1} \cap \Z^{n}$
index the terms of
\[ c(\Pol^d(\mathbb{C}^n)) = \prod_{(i_1,\dots,i_n) \in d\Delta^{n-1} \cap \Z^{n}}
 \left( 1+i_1 x_1+\dots+i_n x_n \right). \]

 The natural group action of $S_n$ on $\Z^n$ restricts to $d\Delta^{n-1} \cap \Z^{n}$.
Its orbits correspond to
length $n$ weak partitions $(d_1, d_2,\dots, d_n)$ of $d$.

Note that (for the sake of Proposition \ref{prop_enumerate_orbits})
the partitions indexing these orbits will chosen to be increasing.
To differentiate them from the decreasing ones occuring in the rest of the paper we will refer to
them as \emph{increasing weak partitions}.

For each orbit the product of the terms of $c(\Pol^d(\mathbb{C}^n))$ corresponding to the lattice points in
the orbit can be expressed in elementary symmetric polynomials $e_1,\dots,e_n$ in $x_1,\dots,x_n$.
For any increasing weak partition $(d_1,\dots,d_n) \vdash d$, let
\begin{equation*}\label{eq_orbitterm}
 p_{(d_1,\dots,d_{n})}(e_1,\dots,e_n):=
  \left. \prod_{(i_1,\dots,i_n) \in S_n \cdot (d_1,\dots,d_n)}\left( 1+i_1 x_1+\dots+i_n x_n \right)
    \right|_{\substack{x_1+\dots+x_n \mapsto e_1 \\ x_1x_2+\dots+x_{n-1}x_n \mapsto e_2 \\ \dots}}.
 \end{equation*}
E.g. for $n=3$ and $(d_1 , d_2 , d_3)=(1,1,4) \vdash d=6$
\begin{multline*}
p_{(1,1,4)}(e_1,e_2,e_3)=
\\
  \left( 1+x_1+x_2+4x_3 \right)\left( 1+x_1+4x_2+x_3 \right)(1+4x_1+x_2+x_3)
  \Big|_{\substack{x_1+x_2+x_3 \mapsto e_1 \\ x_1x_2+x_1x_3+x_2x_3 \mapsto e_2 \\ x_1x_2x_3 \mapsto e_3}}
=\\
1+6e_1+9e_1^2+9e_2+4e_1^3+9e_1e_2+27e_3.
\end{multline*}

\medskip

The polynomials $p_{d_1,\dots,d_n}(e_1,\dots,e_n)$ may be quite different from orbit to orbit
(even their $\{e_1,\dots,e_n\}$-degree can vary).
We will group the orbits---according to the multiciplities of the members of their corresponding
increasing weak partitions---such that on each group we have a uniform description
of the $p_{d_1,\dots,d_n}(e_1,\dots,e_n)$'s:

For each \emph{orbit-type-vector} $u=(u_1,\dots,u_s)$  with $u_t \in \N_{>0}$ and $\sum_{i=1}^s u_i=n$, let
\[ O_u(d):= \left\{ \left( d_1^{u_1},\dots,d_s^{u_s} \right) \vdash d : 0 \le d_1 < \dots < d_s \right\} \]
denote a (possibly empty) subset of orbits.
Elements of $O_u(.)$ we will refer to as \emph{increasing weak partitions of type $u$}.
The corresponding orbits we will call \emph{orbits of type $u$}.

\begin{example}
For $n=4$ and $d=8$ the $O_u(8)$'s are
\begin{align*}
  &O_{(1,1,1,1)}(8)=\left\{ (0,1,2,5),(0,1,3,4) \right\}, \\
  &O_{(2,1,1)}(8)=\left\{ (0,0,1,7),(0,0,2,6),(0,0,3,5),(1,1,2,4) \right\}, \\
  &O_{(1,2,1)}(8)=\left\{ (0,1,1,6),(0,2,2,4),(1,2,2,3) \right\}, \\
  &O_{(1,1,2)}(8)=\left\{ (0,2,3,3) \right\}, \\
  &O_{(2,2)}(8)=\left\{ (0,0,4,4), (1,1,3,3) \right\}, \\
  &O_{(3,1)}(8)=\left\{ (0,0,0,8),(1,1,1,5) \right\}, \\
  &O_{(1,3)}(8)=\emptyset, \\
  &O_{(4)}(8)=\left\{ (2,2,2,2) \right\}. 
\end{align*}
\end{example}

This grouping of orbits gives us a grouping of terms of $c\big(\!\Pol^d(\C^n)\big)$:
For $d \in \N$ and orbit-type-vector $u=(u_1,\dots,u_s)$ let
\begin{equation*}
  R_u(d,e_1,\dots,e_n):=\prod_{(d_1,\dots,d_n) \in O_u(d)} p_{(d_1,\dots,d_n)}(e_1,\dots,e_n).
\end{equation*}
We will call these products \emph{orbit-type-terms}.
Then 
\begin{equation}\label{eq_cPold_intoorbittypes}
c(\Pol^d(\mathbb{C}^n))=\prod_{\substack{u=(u_1,\dots,u_s) \\ |u|=n}}
R_u(d,e_1,\dots,e_n).
\end{equation}
Note that this factorization does not depend on $d$.

\bigskip

To verify the upper bound for the degrees of the coefficients $q_{\nu}(d) \in \Q[d]$ of
$c\big(\!\Pol^d(\C^n)\big)=\sum_\nu q_{\nu}(d) e_\nu$ we will investigate the $d$-dependence of
the orbit-type-terms $R_u(d,e_1,\dots,e_n)$.

Let us abbreviate by $e$ the list of elementary symmetric polynomials $e_1,\dots,e_n$.
Using our theory of rising products, we will---in Section \ref{sec_factorsarequasipolynomials}---show that
the orbit-type-terms
$R_u(d,e)$ 
are \emph{quasi-formal power series}:
we will prove that there exists a period $M(u) \in \N_{>0}$ and
\emph{constituents}
\[ R_{u,\overline{q}_{M(u)}}(d,e) \in \Q[d][[e]] \quad  \left(\overline{q}_{M(u)} \in \Z/M(u)\Z\right)\]
such that
$R_u(d,e)=R_{u,\overline{q}_{M(u)}}(d,e)$ if $\overline{d}_{M(u)}=\overline{q}_{M(u)}$
(where $\overline{d}_{M(u)} \in \Z/M(u)\Z$ denotes the congruence class of $d$).

The reason for writing $c\big(\!\Pol^d(\C^n)\big) \in \Q[d][[e]]$ as a complicated product
\eqref{eq_cPold_intoorbittypes} of
orbit-type-terms is that we can establish a (universal) upper bound for the degrees of the
coefficients of their constituents:
Let
\[ R_{u,\overline{q}_{M(u)}}(d,e)=:\sum_H R_{u,\overline{q}_{M(u)},H}(d)e^H. \]
Then for $u=(u_1,\dots,u_s)$ we have
\[ \deg(R_{u,\bar{q}_{M(u)},H}(d)) \le sH_1+(s+1)H_2+\dots+(s+n-1)H_n
\quad \text{ for every } \bar{q}_{M(u)} \in \Z/M(u)\Z.  \]
 
This way, we will---in Section \ref{see_subdivideN}---be able to subdivide $\left\{ d \in \N \right\}$
into conjugacy classes, each on which we can describe $c\big(\!\Pol^d(\C^n)\big)=\sum_\nu q_{\nu}(d) e_\nu$ as
a finite product of the $R_{u,\bar{q}_{M(u)}}(d,e)$'s whose coefficients have degrees of known
upper bounds.
Then
the following straightforward Lemma \ref{lemma_degrees4products} will provide an
upper bound for $\deg(q_\nu(d))$.

\begin{lemma}\label{lemma_degrees4products}
Let
\[ P(d,e)=\sum_{H}P_H(d) e^H \in \mathbb{Q}[d][[e]] \text{ and } Q(d,e)=\sum_{H}Q_H(d) e^H \in
\mathbb{Q}[d][[e]], \]
and assume that there exists a linear form $W: \Z^{\infty} \to \Z$ such that for every $H \in \N^{\infty}$
$\deg(P_H(d)), \deg(Q_H(d)) \le W(H)$. Then for the coefficients $R_H(d) \in \Q[d]$ in the product
\[ P \cdot Q (d,e)= \sum_{H} R_H(d) e^H \in \mathbb{Q}[d][[e]] \]
their degrees are bounded by the same linear form: $\deg(R_H(d)) \le W(H)$.
\end{lemma}

\begin{remark}
  Quasi-formal power series is a term we invented based on the well-established notion of
  quasi-polynomials.

  The most well-known example of a quasi-polynomial is probably the Ehrhart quasi-polynomial
  calculating the number of points in $dP \cap \Z^n$ for a rational convex $n$-polytope
  $P$ \cite{erhart1967polyedres}.
  More aligned with our setup, 
  the theory of P-partitions (see \cite[Thm.~4.5.8]{stanley1997enumerative_vol1}) or
  Stapledon's equivariant Ehrhart theory \cite{stapledon2011equivehrhart} shows that
  the number of orbits of the $S_n$-action on $d\Delta^{n-1} \cap \Z^{n}$ and 
  the number of those orbits which are of type $(1^n)$ are also quasi-polynomials in
  $d$.
  
  We did not find a way to exploit these results to prove the quasi-formal power
  series property of the orbit-type-terms.
  Rather, our method uses an explicit enumeration of the orbits of $O_u(d)$,
  Proposition \ref{prop_enumerate_orbits}.
  This naturally implies that for any orbit-type-vector $u$ the number of increasing weak
  partitions in $O_u(d)$ is a 
  quasi-polynomial in $d$, see Remark \ref{rmrk_numberoftypeuorbits_quasipol}.
\end{remark}

\subsection{The orbit-type-terms of \texorpdfstring{$c\big(\!\Pol^d(\C^n)\big)$}{c(Pold(Cn))} are quasi-formal power series}
  \label{sec_factorsarequasipolynomials}
  Given an orbit-type-vector $u=(u_1,\dots,u_s)$
  and any increasing weak partition $(d_1^{u_1},\dots,d_s^{u_s})$ of this type
  the $S_n$-action on $S_n \cdot (d_1^{u_1},\dots,d_s^{u_s})$ factors through the free quotient
  group action of $S_n/S_{u_1} \times \dots \times S_{u_s}$.

As a consequence, there is a uniform description of the terms $p_{(d_1^{u_1},\dots,d_s^{u_s})}(e)$
that works for every $(d_1^{u_1},\dots,d_s^{u_s})$ of type $u$:
there exists a polynomial 
$P_u(d_1,\dots,d_s,e) \in \Q[d_1,\dots,d_s,e]$
such that for every $(d_1^{u_1},\dots,d_s^{u_s}) \in O_u(d)$
\[ P_u(d_1,\dots,d_s,e)=p_{(d_1^{u_1},\dots,d_s^{u_s})}(e) .\]
We will call these polynomials \emph{orbit-terms}.
E.g. for $n=4$ and $u=(2,2)$ the corresponding orbit-term is
\begin{multline*}\label{eq_22orbittermexample}
  P_{(2,2)}(d_1,d_2,e_1,e_2,e_3,e_4):=
  p_{\left(d_1^{2},d_2^{2}\right)}(e_1,e_2,e_3,e_4)=\\
\begin{aligned}
 & (1+d_1x_1+d_1x_2+d_2x_3+d_2x_4)(1+d_1x_1+d_2x_2+d_1x_3+d_2x_4)\\
 & (1+d_1x_1+d_2x_2+d_2x_3+d_1x_4)(1+d_2x_1+d_1x_2+d_1x_3+d_2x_4)\\
 & (1+d_2x_1+d_1x_2+d_2x_3+d_1x_4)(1+d_2x_1+d_2x_2+d_1x_3+d_1x_4)
  \Big|_{\substack{x_1+\dots+x_4 \mapsto e_1 \\ \dots}}=
  \end{aligned}\\
  1+(3d_1+3d_2)e_1+(3d_1^2+9d_1d_2+3d_2^2)e_1^2+(d_1^2-4d_1d_2+d_2^2)e_2+\dots\\
  +(-d_1^6+2d_1^5d_2+d_1^4d_2^2-4d_1^3d_2^3+d_1^2d_2^4+2d_1d_2^5-d_2^6)e_1^2e_4
  \end{multline*}

A straightforward observation that will be important for our degree calculations is that
for every orbit-type-vector $u=(u_1,\dots,u_s)$ and exponent vector $H=(H_1,\dots,H_n)$
\begin{equation}\label{eq_orbittermdegreeestimate}
\deg\left(P_{u,H}(d_1,\dots,d_s)\right)=H_1+2H_2+\dots+nH_n,
\end{equation}
where
\[ P_{u}(d_1,\dots,d_s,e)=\sum_H P_{u,H}(d_1,\dots,d_s) e^H. \]

\medskip

We want to investigate the $d$-dependence of the orbit-type-terms $R_{u}(d,e)$ analogously to
how we did it for $c(\Pol(\C^n))$: in the proof of Proposition \ref{prop_cPoldCn_monomsymm} 
we wrote $c(\Pol(\C^n))$ as an $(n-1)$-long sequence of rising products
\eqref{eq_cpoldcn_asseriesofproducts} that we could analyze product-by-product.
Hence, our goal is to obtain an expression for the orbit-type-terms, Corollary
\ref{cor_Ruasproducts}, similar to \eqref{eq_cpoldcn_asseriesofproducts}.
To do so, we need to enumerate orbits of $O_u(d)$.

\begin{proposition}\label{prop_enumerate_orbits}
  For every orbit-type-vector $u=(u_1,\dots,u_s)$ 
  there are linear polynomials $ F_{j}(d,t_1,\dots,t_{j-1}) \in \Z\left[d,t_1,\dots,t_{j-1} \right]$ and integers
  $N_j \in \N_{>0}$ $(j=1,\dots,s-1)$, such that for every $d \in \N$
  the elements of $O_u(d)$ can be listed as follows
  \begin{align*}
    &\left\{  \left( d_1^{u_1},\dots,d_s^{u_s} \right) : \sum_{j=1}^s u_j d_j =d \text{ and }
    d_j < d_{j+1} \right\}=\\
    &\begin{aligned}
      \Bigg\{ & \Big( t_1^{u_1},
	  \dots,
	  (t_1+t_2+\dots+t_{s-1}+s-2)^{u_{s-1}},
	\frac{1}{u_s}\big( d-\sum_{j=1}^{s-1} u_j(t_1+\dots+t_{j}+j-1) \big)^{u_s} \Big) : \\
	& \quad t_1=0,\dots,\left\lfloor F_1(d)/N_1 \right\rfloor,\\ 
	&\quad t_j=0,\dots,\left\lfloor F_j(d,t_1,\dots,t_{j-1})/N_j \right\rfloor, \\ 
	&\quad t_{s-1}=0,\dots,\left\lfloor F_{s-1}(d,t_1,\dots,t_{s-2})/N_{s-1}  \right\rfloor, \\ 
	& \quad u_s \mid  d-\sum_{j=1}^{s-1} u_j(t_1+\dots+t_{j}+j-1)
      \Bigg\}.
    \end{aligned}
  \end{align*}
\end{proposition}

\begin{proof}
  In the $s=1$, $u=(n)$ case the Proposition is trivial, 
  \[ O_u(d)=\left\{ \left( \frac{d}{n}^{n}\right) : n \mid d \right\}, \]
  so let us assume that $s>1$.

  For every $j=1,\dots,s$ let
  \[ u_{\ge j}:=(u_j,\dots,u_s) \text{ and } N_j:=\left| u_{\ge j} \right|. \]
  Then for every $d, t_1,\dots, t_s \in \N$---
  substracting the vector $\left( t_1^{u_1},\left( t_1+1 \right)^{u_2+\dots+u_s} \right)$, we see that---
    \begin{multline}\label{eq_Oudequivalence1}
    \left( t_1^{u_1},t_1+t_2+1^{u_2},\dots,t_1+\dots+t_s+s-1^{u_s} \right) \in
    O_u(\underbrace{d}_{D_1(d)})
    \Longleftrightarrow
    \\
    \left( t_2^{u_2},t_2+t_3+1^{u_3},\dots,t_2+\dots+t_s+s-2^{u_s} \right) \in
    O_{u_{\ge 2}}\Big(\underbrace{d-t_1 \sum_{i=1}^s  u_i  -\sum_{i=2}^s u_i}_{D_2(d,t_1)}\Big).
  \end{multline}
  If $s>2$, we can continue in a similar fashion: substracting the vector
  $\left( t_{j-1}^{u_{j-1}}, t_{j-1}+1 ^{u_j+\dots+u_s}\right)$ from the element of
  $O_{u_{\ge j-1}}(D_{j-1}(d,t_1,\dots,t_{j-2}))$ in the previous step, we get
  that for every $d,t_1,\dots,t_s \in \N$ and $2 \le j \le s$, 
  \eqref{eq_Oudequivalence1} is further equivalent to
  \begin{equation}\label{eq_Oudequivalence2}
    \left( t_{j}^{u_j},\dots,t_j+\dots+t_s+s-j^{u_s}\right) \in
    O_{u_{\ge j}}\Big( D_j(d,t_1,\dots,t_{j-1}) \Big) ,
  \end{equation}
  where $D_j(d,t_1,\dots,t_{j-1}):=D_{j-1}(d,t_1,\dots,t_{j-2})-N_{j-1}t_{j-1}
  -N_j$.
  In particular, from the $j=s$ case we obtain that for every $d \in \N$
  \begin{equation*}
    O_u(d)=\left\{ \left( t_1^{u_1},\dots,t_1+\dots+t_s+s-1^{u_s} \right) :
    t_1,\dots,t_s \in \N \text{ and }
  \left( t_s^{u_s} \right) \in O_{u_{\ge s}}(D_s(d,t_1,\dots,t_{s-1})) \right\}.
  \end{equation*}

  As $u_{\ge s}=(u_s)$, $t_s=D_s(d,t_1,\dots,t_{s-1})/u_s$ for every
  $(t_1^{u_1},\dots,t_1+\dots+t_{s}+s-1^{u_s}) \in O_u(d)$. This leads to our final description,
  \begin{equation*}
    O_u(d)=\left\{ \left( t_1^{u_1},\dots,t_1+\dots+t_s+s-1^{u_s} \right) :
      t_1,\dots,t_{s-1} \in \N \text{ and }
      t_s=\frac{D_s(d,t_1,\dots,t_{s-1})}{u_s} \in \N
    \right\}.
  \end{equation*}
  Note that, by definition,
  \[ D_s(d,t_1,\dots,t_{s-1})=
    d-\sum_{j=1}^{s-1}\left( u_j(t_1+\dots+t_j+j-1)\right) -u_s(t_1+\dots+t_{s-1}+s-1)
  \]
  therefore $u_s \mid D_s(d,t_1,\dots,t_{s-1})$ is equivalent to the divisibility condition
  of the Proposition.

  \medskip

  In the $s>2$ case we can complement the above description,
  $t_s=D_s(d,t_1,\dots,t_{s-1})/u_s \in \N$ with the ranges for the possible values of the $t_j$'s: 
  As we have discussed, $(t_1^{u_1},\dots,t_1+\dots+t_{s}+s-1^{u_s}) \in O_u(d)$
  implies \eqref{eq_Oudequivalence2}. In particular,
  $O_{u_{\ge j+1}}(D_{j+1}(d,t_1,\dots,t_j)) \neq \emptyset$ for every $j=1,\dots,s-1$.
  As $\left( 0^{u_{j+1}},1^{u_{j+2}},\dots,s-(j+1)^{u_s} \right)$ is the smallest possible
  increasing weak partition of orbit-type $u_{\ge j+1}$,
  this gives that
  \begin{multline*}
    D_{j+1}(d,t_1,\dots,t_j)=
    D_j(d,t_1,\dots,t_{j-1})-N_j t_j - N_{j+1} \ge
    \sum_{i=j+1}^s u_i(i-(j+1)) \Longleftrightarrow
    \\
    \left\lfloor \frac{D_j(d,t_1,\dots,t_{j-1}) - \sum_{i=j}^s u_i(i-j)}{N_j} \right\rfloor \ge
    t_j.
  \end{multline*}
  Introducing 
  \[ F_j(d,t_1,\dots,t_{j-1}):=D_j(d,t_1,\dots,t_{j-1}) - \sum_{i=j}^s(i-j) u_i \quad
  (j=1,\dots,s-1) \]
  for the numerators, we finish the proof.
\end{proof}

\begin{remark}\label{rmrk_Fsdef}
  Given an orbit-type-vector $u=(u_1,\dots,u_s)$ let
  \[ F_s(d,t_1,\dots,t_{s-1}):=
  d-\sum_{j=1}^{s-1} u_j(t_1+\dots+t_{j}+j-1) -u_s (t_1+\dots+t_{s-1}+s-1). \]
  Recall, that this polynomial is the same as $D_s(d,t_1,\dots,t_{s-1})$ in the proof of 
  Proposition \ref{prop_enumerate_orbits}, and hence, as we have mentioned there,
  \[ u_s \mid  d-\sum_{j=1}^{s-1} u_j(t_1+\dots+t_{j}+j-1)
    \Longleftrightarrow
  u_s \mid F_s\left( d,t_1,\dots,t_{s-1} \right). \]
  We will continue to use the right-hand side to express the divisibility condition of Proposition
  \ref{prop_enumerate_orbits}.
\end{remark}

\begin{example}
  For the orbit-type-vector $u=(3,1,1,2)$
  \begin{multline*}
    O_{\left(3,1,1,2\right)}(d)=\Bigg\{ \left( t_1^3,t_1+t_2+1^1,t_1+t_2+t_3+2^1,
      \frac{1}{2} \left(d-5t_1-2t_2-t_3-3 \right)^2 \right):
      \\
      \begin{aligned}
	&t_1=0,\dots,\left\lfloor \frac{d-9}{7} \right\rfloor,
	t_2=0,\dots,\left\lfloor \frac{d-7t_1-9}{4} \right\rfloor,\\
      &t_{3}=0,\dots,\left\lfloor \frac{d-7t_1-4t_2-9}{3} \right\rfloor,
      2 \mid d-7t_1-4t_2-3t_3-9
    \Bigg\}.
  \end{aligned}
  \end{multline*}
\end{example}

\bigskip

Now, we can start to investigate the $d$-dependence of the orbit-type-terms $R_u(d,e)$.
First, let us see what happens for the orbit-type-vectors of length one.
\begin{example}\label{ex_Rforlength1u}
  For $u=(n)$, $R_u(d,e)$ is a quasi-polynomial with period $n$: there are polynomials
  (its constituents)
  $R_{u,\bar{q}_n}(d,e) \in \Q[d,e]$ $\left( \bar{q}_n \in \Z/n\Z \right)$ such that
  \[ R_u(d,e)=R_{u,\overline{q}_n}(d,e) \quad \text{ if } \overline{d}_n=
  \overline{q}_n \left( \in \Z/n\Z \right). \]
  The constituents are
  \[ R_{u,\overline{q}_n}(d,e)=
    \begin{cases}
      1+\frac{d}{n}e_1 & \text{ if } \overline{q}_n=\overline{0}_n,
      \\[8pt]
      1  & \text{ otherwise.}
    \end{cases}
    \]
\end{example}

To obtain a similar expression for the orbit-type-vector $u=(u_1,\dots,u_s)$ in the $s>1$ case,
first, let us note that Proposition \ref{prop_enumerate_orbits} provides the following
factorization of the orbit-type-terms.
\begin{corollary}\label{cor_Ruasproducts}
  For every orbit-type-vector $u=(u_1,\dots,u_s)$ of length $s>1$
  \begin{multline}
    \label{eq_Rufactorized}
  R_u(d,e)=
  \prod_{t_1=0}^{\left\lfloor F_1(d)/N_1 \right\rfloor} \dots
  \prod_{\substack{t_{s-1}=0 \\ u_s \mid F_s(d,t_1,\dots,t_{s-1})}}^{\left\lfloor F_{s-1}(d,t_1,\dots,t_{s-2})/N_{s-1} \right\rfloor}
  \\
  P_u\left(t_1,\dots,t_1+\dots+t_{s-1}+s-2,
  \frac{1}{u_s}\big( d-\sum_{j=1}^{s-1} u_j(t_1+\dots+t_{j}+j-1)   \big) ,e\right),
\end{multline}
where $P_u(d_1,\dots,d_s,e)$ is the orbit-term, the polynomials $F_j(d,t_1,\dots,t_{j-1})$ and integers $N_j$ for $u$ were introduced in 
Proposition \ref{prop_enumerate_orbits} and Remark \ref{rmrk_Fsdef}.
\end{corollary}

This is a factorization similar to what we had for $c\big(\!\Pol^d(\C^n)\big) \in
\Q[d][[x_1,\dots,x_n]]^{S_n}$ in \eqref{eq_cpoldcn_asseriesofproducts}. 
The divisibility condition $u_s \mid F_s(d,t_1,\dots,t_{s-1})$ and that
the upper bounds of the products are given by floors of polynomials,
$\left\lfloor F_j(d,t_1,\dots,t_{j-1})/N_j \right\rfloor$ are major differences though, and will
lead to the quasi-formal power series behaviour of the orbit-type-terms:

\begin{proposition}\label{prop_modRu}
  For every orbit-type-vector $u=(u_1,\dots,u_s)$ the orbit-type-term is a
  \emph{quasi-formal power series} with period dividing $M(u)=\prod_{j=1}^s (u_j+\dots+u_s)$:
  there are formal power series (its \emph{constituents}) $R_{u,\overline{q}_{M(u)}}(d,e) \in
  \Q[d][[e]] \, \big(\overline{q}_{M(u)} \in \Z/M(u)\Z\big)$
  such that
  \[ R_{u}(d,e)=R_{u,\bar{q}_{M(u)}}(d,e)
  \quad \text{ if } \, \overline{d}_{M(u)}=\overline{q}_{M(u)}.\]
  Moreover, let $R_{u,\bar{q}_{M(u)}}(d,e)=:\sum_H R_{u,\bar{q}_{M(u)},H}(d)e^H$.
  Then for every coefficient  $R_{u,\bar{q}_{M(u)},H}(d) \in \Q[d]$
  \[ \deg(R_{u,\bar{q}_{M(u)},H}(d)) \le sH_1+(s+1)H_2+\dots+(s+n-1)H_n. \]
\end{proposition}

The proof of Proposition \ref{prop_modRu} will be inductive: starting from the innermost
product of \eqref{eq_Rufactorized}, we go outwards to reach $R_u(d,e)$.
The induction step will be provided by the following lemma.
It can be regarded as a 
natural extension of Theorem \ref{thrm_polynomialityofcoeffs_generalsetup} and Proposition
\ref{prop_stepLW}
to rising products where the formal power series terms of the product depend on a
congruence class.

\begin{lemma}\label{lemma_congruenceclasspolys}
  Suppose that the formal power series $Q_{\overline{m}_M}(d,t,e) \in \Q[d,t][[e]] \,
  (\overline{m}_M \in \Z/M\Z)$ satisfy $Q_{\overline{m}_M}(d,t,0)=1$,
  where $d$ and $e$ may denote a list of variables: $d=(d_1,\dots,d_s)$ and
  $e=(e_1,\dots,e_n)$.
  Let $F(d) \in \Z[d]$ be a linear polynomial, $U \in \Z$ and $N \in \N_{>0}$.
  \begin{enumerate}[label=\roman*)]
    \item\label{item_existcongruenceclasspolys} 
      Then there exist formal power series
      $R_{\overline{q}_{MN}}(d,e) \in \Q[d][[e]] \, (\overline{q}_{MN} \in \Z/MN\Z)$ such that
      for every $d$ with $F(d) \ge 0$
      \[ R_{\overline{F(d)}_{MN}}(d,e)=
      \prod_{t=0}^{\left\lfloor \frac{F(d)}{N} \right\rfloor} Q_{\overline{F(d)-Ut}_M}(d,t,e). \]
    \item\label{item_degreecongruenceclasspolys}
      Let $Q_{\bar{m}_M}(d,t,e)=:\sum_E Q_{\bar{m}_M,E}(d,t)e^E$.
      If there exists a linear form $W:\Z^{\infty} \to \Z$ such that for every
      $\bar{m}_M \in \Z/M\Z$ and exponent $E \in \N^{\infty}$ $\deg(Q_{\bar{m}_M,E}(d,t)) \le
      W(E)$, then for every $\bar{q}_{MN} \in
      \Z/(MN)\Z$ and $H \in \N^{\infty}$
      \[ \deg(R_{\bar{q}_{MN},H}(d)) \le W(H)+|H|, \]
      where $R_{\bar{q}_{MN}}(d,e)=:\sum_H  R_{\bar{q}_{MN},H}(d)e^H$.
  \end{enumerate}
\end{lemma}

\begin{proof}[Proof of Lemma \ref{lemma_congruenceclasspolys}]
  Let us fix a congruence class $\overline{q}_{MN} \in \Z/MN\Z$. To prove
  \ref{item_existcongruenceclasspolys} we have to show that there exists a formal power series
  $R_{\overline{q}_{MN}}(d,e) \in \Q[d][[e]]$ such that
  \begin{equation}\label{eq_powerseries_for_acongclass}
  R_{\bar{q}_{MN}}(d,e)=
    \prod_{t=0}^{\left\lfloor \frac{F(d)}{N} \right\rfloor} Q_{\overline{F(d)-Ut}_M}(d,t,e)
  \quad \text{ if } \overline{F(d)}_{MN}=\overline{q}_{MN}.
\end{equation}

  For the congruence class $\overline{q}_{MN} \in \Z/MN\Z$ there exist unique integers
  $r \in \left\{ 0,\dots,M-1 \right\}$ and $p \in \left\{ 0,\dots,N-1 \right\}$ such that
  $\overline{p+N r}_{MN}=\bar{q}_{MN}$.
  Then we can define a linear polynomial $\delta(d)=\frac{F(d)-p-Nr}{MN} \in \Q[d]$---
  integer valued for $d$ with $\overline{F(d)}_{MN}=\overline{q}_{MN}$---for which
  \begin{equation*}\label{eq_Fdexpansion}
    F(d)=p +Nr+\delta(d)MN
    \quad \text{ if }  \overline{F(d)}_{MN}=\overline{q}_{MN}.
  \end{equation*}

  This means if $\overline{F(d)}_{MN}=\overline{q}_{MN}$ we can group the
  indices $t=0,\dots,\lfloor \frac{F(d)}{N} \rfloor=r+\delta(d)M$ as
  \begin{equation}\label{array_tintogroups}
    \begin{gathered}
    \begin{array}{llcl}
      0 & 1 & \dots & r
    \end{array}
    \\
    \text{ and }
    \\
    \arraycolsep=1.5pt
    \begin{array}{lclcl}
      r+1 & \dots & r+1+\eta & \dots & r+1+M-1 \\
      \vdots & & \vdots & & \vdots \\
      r+1+M\tau & \dots & r+1+\eta+M\tau & \dots & r+1+M-1+M\tau \\
      \vdots & & \vdots & & \vdots \\
      r+1+M(\delta(d)-1) & \dots & r+1+\eta+M(\delta(d)-1) & \dots & r+1+M-1+M(\delta(d)-1),
    \end{array}
  \end{gathered}
\end{equation}
  and correspondingly write
  \begin{multline*}
    \prod_{t=0}^{\left\lfloor \frac{F(d)}{N} \right\rfloor} Q_{\overline{F(d)-Ut}_M}(d,t,e)=
    \\
    \underbrace{\prod_{t=0}^{r} Q_{\overline{F(d)-Ut}_M}(d,t,e)}_{E(d,e)}\cdot
    \underbrace{\prod_{\tau=0}^{\delta(d)-1}
    \underbrace{\prod_{\eta=0}^{M-1} Q_{\overline{F(d)-U(r+1+\eta+M\tau)}_M}(d,r+1+\eta+M\tau,e)}_{B(d,\tau,e)}}_{G(d,e)}.
  \end{multline*}

  The term $E(d,e)$
  is a finite product of elements of $\Q[d][[e]]$, therefore
  $E(d,e)=: \sum_{H} E_H(d) e^H \in \Q[d][[e]]$ and, by Lemma
  \ref{lemma_degrees4products}, $\deg(E_H(d)) \le W(H)$ for every exponent $H$.

  As for the second term, the vector of congruence classes corresponding to the $\tau$-th row of
  diagram \eqref{array_tintogroups}
  \begin{multline*}
    \left(\overline{F(d)-U(r+1+\eta+M\tau)}_M\right)_{\eta=0,\dots,M-1}=
    \\
    \left(\overline{p+Nr+\delta(d)MN-U(r+1+\eta+M\tau)}_M\right)_{\eta=0,\dots,M-1}=:
    \left( W_{\eta} \right)_{\eta=0,\dots,M-1} \subset \left(\Z/M\Z\right)^M
  \end{multline*}
  does not depend on $\tau$ nor on $\delta(d)$.
  Therefore, there exists a (single) formal power series
  \[ B(d,\tau,e):= \prod_{\eta=0}^{M-1} Q_{W_\eta}(d,r+1+\eta+M\tau,e) \in \Q[d,\tau][[e]] \]
  as above. As the substitution $t \mapsto r+1+\eta+M\tau$ is linear in $\tau$, Lemma
  \ref{lemma_degrees4products} implies that the coefficients
  $B_E(d,\tau)$ in the expansion
  $B(d,\tau,e)=:\sum_E B_E(d,\tau)e^E$ have degree $\deg(B_E(d)) \le  W(H)$.

  For $d$'s with $F(d) \ge 0$ we have $\delta(d)-1 \ge -1$.
  Therefore, we can use Theorem \ref{thrm_polynomialityofcoeffs_generalsetup} to deduce that for
  such $d$'s the rising product 
  \[ S[B,\delta(d)-1](d,e)= \prod_{\tau=0}^{\delta(d)-1} B(d,\tau,e)=
  \sum_{H} S[B,\delta(d)]_H(d) e^H  \]
  is described by a formal power series in $\Q[d][[e]]$.
  As $\delta(d)-1$ is linear, Proposition \ref{prop_stepLW} shows that
  $\deg(S[B,\delta(d)-1]_H(d)) \le W(H)+|H|$.

  Finally, let $R_{\overline{q}_{MN}}(d,e):=E(d,e) S[B,\delta(d)-1](d,e) \in \Q[d][[e]]$.
  Then, \eqref{eq_powerseries_for_acongclass} holds by definition, and Lemma
  \ref{lemma_degrees4products} gives that the degree of its coefficients have the desired upper
  bound.
\end{proof}

\begin{proof}[Proof of Proposition \ref{prop_modRu}]
  The $s=1$, $u=(n)$ case is covered by Example \ref{ex_Rforlength1u}. So let $u=(u_1,\dots,u_s)$
  be an orbit-type-vector of length $s>1$.

  For every congruence class $\overline{q}_{u_s} \in \Z/u_s\Z$ define the polynomial
  \begin{multline}\label{eq_Q0def}
    Q^{(0)}_{\bar{q}_{u_s}}(d,t_1,\shorterdots,t_{s-1},e)=\sum_H Q^{(0)}_{\bar{q}_{u_s},H}(d,t_1,\shorterdots,
    t_{s-1})e^H:=
    \\
    \begin{cases}
      P_u\Big(t_1,\shorterdots,t_1+\shorterdots+t_{s-1}+s-2,
	\frac{1}{u_s}\big( d-\sum\limits_{j=1}^{s-1} u_j(t_1+\dots+t_{j}+j-1) \big),
      e\Big) & \text{ if } \overline{q}_{u_s}=\overline{0}_{u_s}, \\[8pt]
      1 & \text{ otherwise},
    \end{cases}
  \end{multline}
  where $P_u(d_1,\dots,d_s,e)$ is the orbit-term for $u$.
  Then, by Corollary \ref{cor_Ruasproducts},
  \begin{equation}\label{eq_orbittypeterm_asproduct}
    R_u(d,e)=
    \prod_{t_1=0}^{\left\lfloor \frac{F_{1}(d)}{N_1} \right\rfloor} \dots
    \prod_{t_{s-1}=0}^{\left\lfloor \frac{F_{s-1}(d,t_1,\dots,t_{s-2})}{N_{s-1}} \right\rfloor}
    Q^{(0)}_{\overline{F_s(d,t_1,\dots,t_{s-1})}_{u_s}}(d,t_1,\dots,t_{s-1},e).
  \end{equation}

  An iterative use of Lemma \ref{lemma_congruenceclasspolys} will provide further sequences of
  formal power series
  \begin{multline}\label{eq_Qjsequences}
    Q^{(j)}_{\bar{q}_{M_{s-j} (u)}}(d,t_1,\dots,t_{s-j-1},e)
    \in \Q[d,t_1,\dots,t_{s-j-1}][[e]]
    \\
    \quad \Big( \bar{q}_{M_{s-j} (u)} \in \Z/M_{s-j} (u)\Z \text{ and } j=0,\dots,s-1 \Big),
  \end{multline}
  the $j$-th sequence indexed by congruence classes modulo
  $M_{s-j} (u):=\prod_{i=s-j}^s (u_i+\dots+u_s)$, such that for every $d,t_1,\dots,t_{s-j-2}$
  with $F_{s-j-1}(d,t_1,\dots,t_{s-j-2}) \ge 0$
  \begin{multline}\label{eq_Qjproperty}
  Q^{(j+1)}_{\overline{F_{s-j-1}(d,t_1,\dots,t_{s-j-2)})}_{M_{s-j-1}(u)}}(d,t_1,\dots,t_{s-j-2},e)=
  \\
  \prod_{t_{s-j-1}=0}^{\left \lfloor \frac{F_{s-j-1}(d,t_1,\dots,t_{s-j-2})}{N_{s-j-1}}\right\rfloor}
  Q^{(j)}_{\overline{F_{s-j}(d,t_1,\dots,t_{s-j-1})}_{M_{s-j}(u)}}(d,t_1,\dots,t_{s-j-1},e)
  \end{multline}
  for every $j=0,\dots,s-2$.

  As for every $d,t_1,\dots,t_{s-2}$ in \eqref{eq_orbittypeterm_asproduct}
  $F_{s-j-1}(d,t_1,\dots,t_{s-j-2}) \ge 0$ holds for every $j=0,\dots,s-2$,
  we can compose equations \eqref{eq_Qjproperty} to get that
  \begin{equation*}
    R_u(d,e)=Q^{(s-1)}_{\overline{F_1(d)}_{M_{1}(u)}}(d,e)=
    \prod_{t_1=0}^{\left\lfloor \frac{F_{1}(d)}{N_1} \right\rfloor} \dots
    \prod_{t_{s-1}=0}^{\left\lfloor \frac{F_{s-1}(d,t_1,\dots,t_{s-2})}{N_{s-1}} \right\rfloor}
    Q^{(0)}_{\overline{F_s(d,t_1,\dots,t_{s-1})}_{u_s}}(d,t_1,\dots,t_{s-1},e).
  \end{equation*}
  By definition, $M_{1}(u)=M(u)$ and $F_1(d)=d-\sum_{i=1}^s (i-1)u_i$, therefore, setting
  \begin{equation}\label{eq_Ruq_def}
    R_{u,\bar{q}_{M(u)}}(d,e):=Q^{(s-1)}_{\overline{q-\sum_{i=1}^s (i-1)u_i}_{M(u)}} \in
    \Q[d][[e]]
 \end{equation}
 finishes the proof of the first part of Proposition \ref{prop_modRu}.
  
 \medskip

  To show the existence of the sequences \eqref{eq_Qjsequences}, let us assume that the
  formal power series
  $Q^{(j)}_{\bar{q}_{M_{s-j} (u)}} \in \Q[d,t_1,\dots,t_{s-j-1}][[e]]$
  $\Big( \overline{q}_{M_{s-j} (u)} \in \Z/M_{s-j} (u)\Z \Big)$
  have been defined.
  Lemma \ref{lemma_congruenceclasspolys} (with $Q_{\bar{r}_M}=Q^{(j)}_{\bar{r}_{M_{s-j}(u)}}$,
    $t=t_{s-j-1}$, $F=F_{s-j-1}$, $N=N_{s-j-1}$ and
  $U=\sum_{i=s-j-1}^s u_i$) shows that there exist formal power series
  \begin{equation*}
    Q_{\bar{q}_{\underbrace{\scriptstyle{M_{s-j}(u)N_{s-j-1}}}_{M_{s-j-1}(u)}}}^{(j+1)}
    (d,t_1,\dots,t_{s-j-2},e)
    \quad \Big( \bar{q}_{M_{s-j-1} (u)} \in \Z/M_{s-j-1}(u)\Z \Big)
  \end{equation*}
  satisfying \eqref{eq_Qjproperty} whenever $F_{s-j-1}(d,t_1,\dots,t_{s-j-2}) \ge 0$.
  Here we used that for $U=\sum_{i=s-j-1}^su_i$ the
  identity
  \[ F_{s-j}(d,t_1,\dots,t_{s-j-1})=F_{s-j-1}(d,t_1,\dots,t_{s-j-2})-U t_{s-j-1}\]
  holds by definition.

  \medskip

  We conclude the proof by verifying the upper bound for $\deg(R_{u,\bar{q}_{M(u)},H}(d))$.
  As the substitutions 
  \[ d_1 \mapsto t_1,\,
    \dots,\, d_{s-1}\mapsto t_1+\dots+t_{s-1}+s-2,
  \, d_s \mapsto \frac{1}{u_s} \big( d-\sum_{j=1}^{s-1} u_j(t_j+\dots+t_{s-1}+j-1)  \big)  \]
   in \eqref{eq_Q0def} are linear,
   the coefficients of $Q^{(0)}_{\bar{q}_{u_s}}$ have the same degrees as those those of the
   orbit-term $P_u(d_1,\dots,d_s,e)$: That is, by \eqref{eq_orbittermdegreeestimate},
   the degrees of the coefficients $ Q_{\bar{q}_{u_s},H}^{(0)} $ are
   given by a linear form,
  \begin{equation*}\label{eq_orbitterm_ints_degreeestimate}
  \deg\left( Q_{\bar{q}_{u_s},H}^{(0)}(d,t_1,\dots,t_{s-1}) \right)
  =H_1+2H_2+\dots+nH_n.
  \end{equation*}
  The iterative application of part \ref{item_degreecongruenceclasspolys} of Lemma
  \ref{lemma_congruenceclasspolys} shows that for the coefficient
  $Q^{(j)}_{\bar{q}_{M_{s-j}(u)},H}$ of $e^H$ in 
  $Q^{(j)}_{\bar{q}_{M_{s-j}(u)}}$ we have
  \[ \deg\left( Q^{(j)}_{\bar{q}_{M_{s-j}(u)},H}(d,t_1,\dots,t_{s-j-1})\right) \le
  \sum_{i=1}^n (iH_i) + j|H| = \sum_{i=1}^n (j+i) H_i. \]
  Finally, by \eqref{eq_Ruq_def}, the $j=s-1$ case proves the degree part of Proposition
  \ref{prop_modRu}.
\end{proof}

\begin{example}\label{ex_biggestorbitRu}
  For the orbit-type-vector $u=\left( 1^{n} \right)$ the corresponding orbit-type-term
  $R_{(1^{n})}(d,e)$ is a quasi-formal power series with period dividing $M(u)=n!$:
  there are formal power series (its constituents) $R_{(1^{n}),\bar{q}_{n!}}(d,e) \in \Q[d][[e]]$
  $( \bar{q}_{n!} \in \Z/n!\Z )$ such that
  \[ R_{(1^{n})}(d,e)=R_{(1^{n}),\bar{q}_{n!}}(d,e) \quad \text{ if } 
  \overline{d}_{n!}=\overline{q}_{n!}. \]

  For their coefficients $R_{(1^{n}),\bar{q}_{n!},H}(d)$ in $R_{(1^{n}),\bar{q}_{n!}}(d,e)=
  \sum_H R_{(1^{n}),\bar{q}_{n!},H}(d) e^H$ we have
  \begin{equation*}
    \deg\left( R_{(1^{n}),\bar{q}_{n!},H}(d) \right) \le nH_1+(n+1)H_2+\dots+(2n-1)H_n.
  \end{equation*}
\end{example}

\begin{remark}\label{rmrk_numberoftypeuorbits_quasipol}
  For every orbit-type-vector $u=\left( 1^{n} \right)$ the number of increasing weak partitions
  in $O_u(d)$ is a quasi-polynomials of degree $s-1$ and period dividing
  $M(u)=\prod_{j=1}^s (u_j+\dots+u_s)$:
  This number is given by the coefficient of $x$ in
  \begin{equation*}
    \prod_{t_1=0}^{\left\lfloor F_1(d)/N_1 \right\rfloor} \dots
    \prod_{\substack{t_{s-1}=0 \\ u_s \mid F_s(d,t_1,\dots,t_{s-1})}}^{\left\lfloor F_{s-1}(d,t_1,\dots,t_{s-2})/N_{s-1} \right\rfloor}
    (1+x),
  \end{equation*}
  where the polynomials $F_j(d,t_1,\dots,t_{j-1})$ and integers $N_j$ for $u$ were introduced in 
  Proposition \ref{prop_enumerate_orbits} and Remark \ref{rmrk_Fsdef}.
  The same way we proved Proposition \ref{prop_modRu}, we can show that this product is a
  quasi-formal power series in $d$ with period dividing $M(u)$.
  In particular, the coefficient of $x$ is a quasi-polynomial in $d$ with period dividing $M(u)$
  whose degree can be computed product-by-product.
\end{remark}

\subsection{A degree estimate for the coefficients of \texorpdfstring{$c\big(\!\Pol^d(\C^n)\big)$}{c(Pold(Cn))} regarded as a quasi-formal power series}\label{see_subdivideN}
The group $\Z/M(u)\Z$---whose elements index the constituents
$R_{u,\bar{q}_{M(u)}} \in \Q[d][[e]]$
---may differ for different orbit-type-vectors $u$.
To make the different orbit-type-terms $R_u(d,e)$ compatible, we pick a single modulus $(n!)$
and index all the constituents with
congruence classes with respect to this modulus:

Luckily, $M(u) \mid n!$ holds for every orbit-type-vector $u$, so we can define
  \[ R_{u,\bar{q}_{n!}}(d,e):= R_{u,\overline{\pi(q)}_{M(u)}}(d,e) \quad \left( \bar{q}_{n!} \in \Z/n!\Z
  \right), \]
  where $\pi: \Z/n!\Z \to \Z/M(u)\Z$ denotes the natural projection.
  This way, for every $\bar{q}_{n!} \in \Z/n!\Z $ we can take the product
  $\prod_{u: |u|=n} R_{u,\overline{q}_{n!}}(d,e) \in \Q[d][[e]]$ and,
  by \eqref{eq_cPold_intoorbittypes},
  \begin{equation*}\label{eq_cPoldCnasprodRu}
  c( \Pol^d(\C^n) )=\prod_{\substack{u=(u_1,\dots,u_s) \\ |u|=n}} R_{u,\overline{q}_{n!}}(d,e)
  \quad \text{ if } \overline{d}_{n!}=\overline{q}_{n!},
  \end{equation*}
  or, in other words,
  \begin{equation}\label{eq_aHdCnascoefficientofprodRu}
    a_{\left(1^{H_1},\dots,n^{H_n}\right)}(d)=
    \coef( e^H, c( \Pol^d(\C^n) ))=
    \coef\Big( e^H, \prod_{\substack{u=(u_1,\dots,u_s) \\ |u|=n}} R_{u,\overline{q}_{n!}}(d,e) \Big)
      \quad \text{ if } \overline{d}_{n!}=\overline{q}_{n!}.
  \end{equation}

  Note that the fact that both sides of \eqref{eq_aHdCnascoefficientofprodRu} are
  polynomials agreeing at infinitely many points immediately implies that
  \eqref{eq_aHdCnascoefficientofprodRu} holds without the congruence assumption.
  As a consequence, we get that the product $\prod_{u: |u|=n} R_{u,\overline{q}_{n!}}(d,e)$ of
  quasi-formal power series $R_{u,\overline{q}_{n!}}(d,e)$ is a formal power series.

  Let $R_{u,\bar{q}_{n!}}(d,e)=\sum_H R_{u,\bar{q}_{n!},H}(d)e^H$. Then, by Proposition
  \ref{prop_modRu},
  \[ \deg(R_{u,\bar{q}_{n!},H}(d)) \le s H_1 + (s+1) H_2 + \dots + (s+n-1) H_n \]
  for every orbit-type-vector $u=(u_1,\dots,u_s)$ and exponent vector $H \in \N^n$.
  As the number of orbit-type-vectors is finite, we can apply Lemma \ref{lemma_degrees4products}
  to deduce that
  \begin{multline*}
    \deg\Bigg( a_{\left( 1^{H_1},\dots,n^{H_n} \right)}(d) \Bigg)=
    \deg\Bigg( \coef \Big( e^H,  \prod_{\substack{u=(u_1,\dots,u_s) \\ |u|=n}}
    R_{u,\bar{q}_{n!}}(d,e) \Big) \Bigg)
    \\
    \le \max_{\substack{u=(u_1,\dots,u_s) \\ |u|=n}} \sum_{i=1}^n (s-1+i)H_i=
    nH_1+(n+1)H_2+\dots+(2n-1)H_n,
  \end{multline*}
  see Example \ref{ex_biggestorbitRu}.
  This finishes the proof.

\bibliographystyle{alpha}  
\bibliography{cpold}

\end{document}